\documentclass[a4paper, 10pt, dvipsnames]{article}

\usepackage[utf8]{inputenc}    
\usepackage[T1]{fontenc}
\usepackage[english]{babel}
\usepackage[normalem]{ulem}
\usepackage{lmodern, dsfont, bm, pifont, mathrsfs}
\usepackage{mwe}

\usepackage[margin=30pt]{subcaption}
\usepackage{float, graphicx, caption, wrapfig, tikz}
\usetikzlibrary{shapes.misc}
\usepackage{amsthm, amsmath, amsfonts, amssymb, mathrsfs, mathtools}
\usepackage[alphabetic]{amsrefs}

\usepackage{hyperref, xcolor, titlesec} 
\hypersetup{colorlinks = true, urlcolor = blue,linkcolor=BrickRed, citecolor=YellowOrange!75!Red, bookmarksopen = true}
\usepackage[nomarginpar]{geometry}
\numberwithin{equation}{section}

\newcommand{\cor}{\color{BrickRed}}

\newcommand{\D}{\mathrm{d}}
\newcommand{\I}{\mathrm{i}}
\newcommand{\scp}[2]{\langle #1,#2\rangle}

\let\O\relax
\newcommand{\O}[1]{\mathcal{O}\left(#1\right)}
\newcommand{\E}[1]{\mathds{E}\left[#1\right]}
\newcommand{\EL}[1]{\mathds{E}\left[#1\middle\vert\bm{\lambda}\right]}
\newcommand{\ELH}[1]{\mathds{E}\left[#1\middle\vert \bm{H_0},\bm{\lambda}\right]}
\newcommand{\unn}[2]{[\![#1,#2]\!]}
\let\Im\relax
\DeclareMathOperator{\Im}{Im}
\let\Re\relax
\DeclareMathOperator{\Re}{Re}
\DeclareMathOperator{\Tr}{Tr}
\DeclareMathOperator{\per}{per}
\newcommand{\specialcell}[2][c]{
  \begin{tabular}[#1]{@{}c@{}}#2\end{tabular}}
\newcommand{\mding}[1]{\text{\ding{#1}}}

\theoremstyle{plain}
\newtheorem{theorem}{Theorem}[section]

\newtheorem{lemma}[theorem]{Lemma}
\newtheorem{corollary}[theorem]{Corollary}
\newtheorem{proposition}[theorem]{Proposition}
\theoremstyle{definition}
\newtheorem{remark}[theorem]{Remark}

\newtheorem{definition}[theorem]{Definition}

\titleformat{\paragraph}[runin]{\bfseries\normalsize}{\theparagraph}{}{}
\titleformat{\subparagraph}[runin]{\itshape\normalsize}{\theparagraph}{0em}{}
\titleformat{\section}[block]{\normalfont\filcenter}{\Large\bf\thesection .}{.5em}{\Large\bf}
\titleformat{\subsection}[block]{\normalfont}{\large \bf \thesubsection .}{.5em}{\large\bf}

\begin{document}
\addtocontents{toc}{\protect\setcounter{tocdepth}{1}}
\title{\textbf{Fermionic eigenvector moment flow}}
\author{L. Benigni\\\vspace{-0.15cm}\footnotesize{\it{LPSM, Université Paris Diderot}}\\\footnotesize{\it{lbenigni@lpsm.paris}}}
\date{}
\maketitle
\begin{abstract}
\small{We exhibit new functions of the eigenvectors of the Dyson Brownian motion which follow an equation similar to the Bourgade-Yau eigenvector moment flow \cite{bourgade2017eigenvector}. These observables can be seen as a Fermionic counterpart to the original (Bosonic) ones. By analyzing both Fermionic and Bosonic observables, we obtain new correlations between eigenvectors: 
\begin{itemize}
\item[$(i)$]The fluctuations $\sum_{\alpha\in I}\vert u_k(\alpha)\vert ^2-{\vert I\vert}/{N}$ decorrelate for distinct eigenvectors as the dimension $N$ grows.
\item[$(ii)$]  An optimal estimate on the partial inner product $\sum_{\alpha\in I}u_k(\alpha)\overline{u_\ell}(\alpha)$ between two eigenvectors is given. 
\end{itemize}
These static results obtained by integrable dynamics are stated for generalized Wigner matrices and should apply to wide classes of mean field models.}
\end{abstract}

\tikzset{cross/.style={cross out, draw=black, minimum size=2*(#1-\pgflinewidth), inner sep=0pt, outer sep=0pt},cross/.default={1pt}}

{\hypersetup{linkcolor=black}\tableofcontents}

\section{Introduction}
Eigenvector statistics of large random matrices were studied extensively as they appear in numerous physical problems and models. In the past decade, a large amount of work was done to prove \emph{universality} of statistics of large random matrices in the sense that they only depend on the symmetry of the matrix but not on the actual entry distribution. While this universality phenomenon was first conjectured and proved for eigenvalue statistics, the same paradigm holds for eigenvectors. 

The integrable model of symmetric random matrices is given by the Gaussian Orthogonal Ensemble which can be seen as a distribution on the space of symmetric random matrices given by
\[
\mathds{P}_{\mathrm{GOE}}(\D H)
=
\frac{1}{Z_N}
\mathrm{e}^{-\frac{N}{4}\Tr H^2}\D H
\]
where $\D H$ is the product Lebesgue measure. Eigenvector statistics for this model are trivial as the distribution is invariant under orthogonal conjugation. The whole normalized eigenbasis is Haar distributed on the orthogonal group and each eigenvector is uniformly distributed on the sphere. Numerous properties of eigenvectors of this ensemble, denoted $u_k$, can be obtained from their exact distribution. Note that eigenvectors are well defined up to a sign (or a phase in the Hermitian case), to remove this ambiguity we can consider them multiplied by an independent centered $\pm 1$ random variable. One has the following probability bound for the extremal coordinate of an eigenvector: for every positive $D$ there exists $C$ such that 
\begin{equation}
\label{eq:integrdeloc}
\mathds{P}\left(
	\text{there exists }k\in\unn{1}{N},\,
	\Vert u_k\Vert_\infty
	\geqslant
	\sqrt{\frac{C\log N}{N}}
\right)\leqslant N^{-D}
\end{equation}
where $\unn{1}{N}\coloneqq\{1,2,\dots,N\}.$

One can also consider the asymptotic distribution of eigenvector entries. Since eigenvectors are uniformly distributed on the sphere and independent, for two sets of indices $I$ and $J$ of fixed cardinality, we have
\begin{equation}\label{eq:integrdist}
\left(
	\sqrt{N}u_k(\alpha)
\right)_{(k,\alpha)\in I\times J}
\xrightarrow[N\rightarrow\infty]{}
\left(
	\mathcal{N}_{k,\alpha}
\right)_{(k,\alpha)\in I\times J}
\end{equation}
where $\mathcal{N}_{k,\alpha}$ are independent standard random variables. Note that while the convergence is stated for a set of indices with fixed cardinality, one can make the set depend on the dimension of the matrix. Indeed, if one considers a single eigenvector, \cite{diaconis1987dozen} proved that the total variation distance between $o(N)$ entries and independent standard normal random variables vanishes. This was refined in \cite{jiang2006how} where the number of eigenvectors could also depend on $N$: this approximation by a Gaussian vector holds if $\vert I\vert=o(\sqrt{N})$ and $\vert J\vert=o(\sqrt{N})$.

Another property is concentration of the eigenvector mass: for two indices $k,\,\ell\in\unn{1}{N}$ and a set of indices $\vert I\vert\gg 1$ then we have for any small $\varepsilon>0$ and large $D>0$,
\[
\mathds{P}\left(
	\left\vert
		\sum_{\alpha\in I}
		u_k(\alpha)^2
		-
		\frac{\vert I\vert}{N}
	\right\vert
	\geqslant 
	N^\varepsilon
	\frac{\sqrt{2\vert I\vert}}{N}
\right)\leqslant N^{-D}
\quad\text{and}\quad
\mathds{P}\left(
	\left\vert
		\sum_{\alpha\in I}
		u_k(\alpha)u_\ell(\alpha)
	\right\vert
	\geqslant
	N^\varepsilon\frac{\sqrt{\vert I\vert}}{N}
\right)\leqslant N^{-D}.
\]
The first bound gives a notion of flatness of the eigenvectors as it states that their mass is evenly spaced on any set of indices while the second bound shows that two eigenvectors are approximately orthogonal when projected onto any subspace of indices. These bounds are a type of quantum unique ergodicity first stated as a conjecture by Rudnick--Sarnak \cite{rudnick1994behaviour} in a different context for eigenfunctions of the Laplace--Beltrami operator on some manifolds.\\[2ex]
\indent The Gaussian Orthogonal Ensemble is the most straightforward example of a \emph{delocalized} system. In the study of correlated quantum systems, it is believed that there are two main behaviors for eigenstates : a \emph{conductor} phase where eigenvectors are delocalized with  strongly correlated eigenvalues and an \emph{insulator} phase where eigenvectors are localized \cite{Anderson} and eigenvalues behaves independently. The GOE is the integrable system corresponding to this \emph{conductor} phase and our paper is a contribution to the study of eigenvectors statistics for more general delocalized systems.

Beyond the GOE integrable system, eigenvectors are no longer uniformly distributed on the sphere. However, one can asks whether eigenvectors behaves in a similar way asymptotically as the dimension grows to infinity. We give now some examples of eigenvector properties in the context of matrices with independent entries which were considered recently and we refer to \cite{orourke2016eigenvectors} for a more complete survey on the subject. The study of extremal coordinates as in \eqref{eq:integrdeloc} has seen a lot of progress for a wide class of models where the first upper bounds were given in \cites{erdos2009semicircle, Tao2011random} using spectral methods which was then improved optimally for bulk eigenvectors in \cite{vu2015random} (this result can be combined with \cite{rudelson2013hanson} for a more general entry distribution). Another quantitative upper bound on eigenvectors was given in \cite{rudelson2015delocalization} using a novel geometric method which also gives an upper bound for non-symmetric random matrices. One can also consider the smallest coordinates and ask if it behaves as if the eigenvectors were uniformly distributed on the sphere: such a lower bound was obtained in \cite{orourke2016eigenvectors}. Another type of eigenvector delocalization was proved in \cite{rudelson2016nogaps} to be universal: eigenvectors have substantial mass on any macroscopic set of coordinates. Finally, one can also consider the asymptotic distribution of eigenvector entries as in \eqref{eq:integrdist}. The works \cites{Tao2012random, knowles2013eigenvector, bourgade2017eigenvector} proved that entries are asymptotically Gaussian and independent. Our contribution is the understanding of correlation between distinct eigenvectors, as described below. \\[2ex]
\indent Denote $W$ a symmetric random matrix with independent entries (up to the symmetry) and consider \linebreak$\lambda_1\leqslant \dots\leqslant \lambda_N$ its ordered eigenvalues and $(u_1,\dots,u_N)$ the associated eigenvectors. For a fixed deterministic sequence of indices $k$ and $I\subset[\![1,N]\!]$ a $N$-dependent set of indices, let
\begin{equation}\label{eq:deffluct}
\tilde{p}_{kk}:=\frac{1}{\sqrt{2\vert I\vert}}\sum_{\alpha\in I}\left(
Nu_k(\alpha)^2-1
\right)
\quad\text{and}\quad
\tilde{p}_{k\ell}:=
\frac{1}{\sqrt{\vert I\vert}}\sum_{\alpha\in I}Nu_k(\alpha)u_\ell(\alpha).
\end{equation}
For asymptotically independent and normally distributed eigenvector entries, we would expect that these random variables converge to a Gaussian random variable in light of the central limit theorem.

While earlier studies on eigenvector distribution give some information on such fluctuations, it is not yet possible to study all joint moments between different entries of different eigenvectors for all Wigner matrices (such as $\mathds{E}u_k^2(1)u_\ell^2(2)$ for instance). The asymptotic Gaussianity of eigenvector entries was proved in \cites{Tao2012random, knowles2013eigenvector, bourgade2017eigenvector} and thus gives some understanding of the distribution of $(\tilde{p}_{k\ell})_{k,\ell}$. The results from \cites{Tao2012random, knowles2013eigenvector} apply to all eigenvector entries but are perturbative: they require models where the moments of the matrix entries match the Gaussian ones up to fourth order. On the other hand, the method of \cite{bourgade2017eigenvector} is non-perturbative but does not give the distribution of entries for \emph{distinct} eigenvectors. As a consequence, the problem of correlations for general eigenvector entries and entry distribution was left open.

The key new ingredient in this paper is the exhibition of a new moment observable that follows the eigenvector moment flow, the dynamics introduced in \cite{bourgade2017eigenvector}. In \cite{bourgade2018random}, another observable involving fluctuations of eigenvectors such as \eqref{eq:deffluct}
was introduced. By gaining information through the observable from \cite{bourgade2018random} and the one from this paper we are able to obtain the Gaussianity and decorrelation of the fluctuations $\tilde{p}_{kk}$. We also expect asymptotic Gaussianity of the mixed overlaps $\tilde{p}_{k\ell}$; here we are able to obtain their asymptotic variance, giving information on their actual size (see Theorem \ref{theo:mainresult}). As will be apparent in Subsection \ref{subsec:1.1}, the observable introduced in our paper actually also give some higher moment information on the $\tilde{p}_{k\ell}$'s.

Finally, we note that the study of fluctuations of eigenvectors were first on the global scale, in the sense that they involved a macroscopic number of eigenvectors. The first result comes from the eigenvectors of large sample covariance matrices in \cite{silverstein1990weak} where it was seen that some form of fluctuations involving all eigenvectors converges weakly to the Brownian bridge. Also, in the case of Gaussian matrices, say symmetric matrices, \cite{donati2012truncations} proved that the process
\begin{equation}\label{eq:averagefluct}
\left(
	\frac{1}{\sqrt{2}}\sum_{\substack{1\leqslant i\leqslant Ns\\1\leqslant j\leqslant Nt}}
	\left(
		\vert u_i(j)\vert^2-\frac{1}{N}
	\right)
\right)_{(s,t)\in[0,1]^2}
\end{equation}
converges to a bivariate Brownian bridge. This result was then generalized to more general model of matrices such as Wigner matrices in \cites{benaych2012universality}. Another form of convergence to the Brownian bridge for Wigner matrices was also proved in \cite{bao2014universality}. In contrast to averages of type \eqref{eq:averagefluct}, our work considers correlations between individual eigenvectors.
\subsection{Main algebraic results: Fermionic observables}\label{subsec:1.1}

Our main result consider dynamics of eigenvectors of random matrices which consists of an Ornstein-Uhlenbeck process on the space of symmetric matrices. The main characteristics of this dynamics is the explicit flow of eigenvectors along the process and the short time to relaxation to the equilibrium measure. For eigenvectors, this measure consists in the Haar measure on orthogonal matrices so that we obtain asymptotic Gaussianity and independence of eigenvectors entries. We now give the definition for the Dyson Brownian motion.

\begin{definition}\label{def:dyson}
Let $B$ be a symmetric $N\times N$ matrix such that $B_{ij}$ for $i<j$ and $B_{ii}/\sqrt{2}$ are standard independent brownian motions and $H_0$ a symmetric matrix. The symmetric Dyson Brownian motion with initial condition $H_0$ is given by the stochastic differential equation 
\begin{equation}\label{eq:dyson}
\D H_s=\frac{1}{\sqrt{N}}\D B_s-\frac{1}{2} H_s\D s.
\end{equation}
Besides, its eigenvalues and eigenvectors have the same distribution at time $s$ than the solution of the following system of coupled stochastic differential equations,
\begin{align}
\mathrm{d}\lambda_k(s)&=\frac{\mathrm{d}\widetilde{B}_{kk}(s)}{\sqrt{N}}+\left(
	\frac{1}{N}\sum_{\ell\neq k}\frac{1}{\lambda_k(s)-\lambda_\ell(s)}-\frac{\lambda_k(s)}{2}
\right)\D s,\label{eq:dysonval}\\
\mathrm{d}u_k^s&=\frac{1}{\sqrt{N}}\sum_{\ell\neq k}\frac{\mathrm{d}\widetilde{B}_{k\ell}(s)}{\lambda_k(s)-\lambda_\ell(s)}u_\ell^s-\frac{1}{2N}\sum_{\ell\neq k}\frac{\mathrm{d}s}{(\lambda_k(s)-\lambda_\ell(s))^2}u_k^s\label{eq:dysonvect}
\end{align}
where $\widetilde{B}$ is an independent copy of $B$.
\end{definition}

The explicit dynamics of eigenvectors \eqref{eq:dysonvect} and the independence of the noises driving \eqref{eq:dysonval} and \eqref{eq:dysonvect}  are key ingredients in our study. While this dynamics is hard to analyze directly, one can look at some observables on eigenvector moments which follows a parabolic equation. This equation, the eigenvector moment flow, was first considered in \cite{bourgade2017eigenvector} to study eigenvectors of generalized Wigner matrices, then in \cite{bourgade2017huang} for sparse matrices and in \cite{benigni2017eigenvectors} for eigenvectors of deformed Wigner matrices. It was also used in a refined way to study band matrices in \cite{bourgade2018random} by introducing observables which follow this eigenvector moment flow. In the rest of the article, we will refer to this observable as \emph{Bosonic} (see \eqref{eq:perfobs}) as a counterpart to the \emph{Fermionic} one we introduce now. 

The Fermionic observables are functions of the fluctuations \eqref{eq:deffluct}. Consider now the $\ell^2$-normalized eigenvectors of $H_s$ as in \eqref{eq:dyson}, $\bm{u}^s=(u_1^s,\dots,u_N^s)$,  and their associated ordered eigenvalues $\bm{\lambda}(s)=(\lambda_1(s)\leqslant\dots\leqslant\lambda_N(s))$. Let $(\mathbf{q}_\alpha)_{\alpha\in I}$ be a family of (non necessarily orthogonal) deterministic fixed vectors.  We can slightly generalize fluctuations by defining for $k\neq \ell$ in $[\![1,N]\!]$ and any $C_0>0$,
\begin{equation}\label{eq:defpkk}
p_{kk}(s)=\sum_{\alpha\in I}\scp{\mathbf{q}_\alpha}{u_k^s}^2-C_0\quad\text{and}\quad p_{k\ell}(s)=\sum_{\alpha\in I}\scp{\mathbf{q}_\alpha}{u_k^s}\scp{\mathbf{q}_\alpha}{u_\ell^s}.
\end{equation}
For $\mathbf{k}=(k_1,\dots,k_n)$ with $k_i$ pairwise distinct indices in $[\![1,N]\!]$, we  define the following $n\times n$ (symmetric) matrix of fluctuations
\begin{equation}\label{eq:defmatrix}
P_s(\mathbf{k})=
\begin{pmatrix}
	p_{k_1 k_1}(s)&p_{k_1 k_2}(s)&\dots &p_{k_1 k_n}(s)\\
	\vdots&\vdots&\dots&\vdots\\
	p_{k_nk_1}(s)&p_{k_nk_2}(s)&\dots&p_{k_n k_n}(s)
\end{pmatrix}.
\end{equation}
The Fermionic observable consists in the expectation of the determinant of our matrix of fluctuations,
\begin{equation}\label{eq:fermion}
f_{s,\bm{\lambda}}^{\mathrm{Fer}}(\mathbf{k})=\mathds{E}\left[
	\det P_s(\mathbf{k})
	\middle\vert \bm{\lambda}
\right],
\end{equation}
where we conditioned on the whole trajectory of eigenvalues from $0$ to $1$.
We use the following notation in order to describe the dynamics followed by $f^{\mathrm{Fer}}$, it consists of replacing the $i$-th coordinated by another index $\ell\notin\{k_1,\dots,k_n\}$:
\[
\mathbf{k}^i(\ell):=(k_1,\dots,k_{i-1},\ell,k_{i+1},\dots,k_n)\quad\text{and}\quad \vert \mathbf{k}\vert=\vert \mathbf{k}^i(\ell)\vert=n.
\]
The functions $f_{s}^{\mathrm{Fer}}$ undergo the following flow.
\begin{theorem}
\label{theo:emf}
Let $(\mathbf{u},\bm{\lambda})$ be the solution to the coupled flows as in Definition \ref{def:dyson} and let $f_{s,\bm{\lambda}}^{\mathrm{Fer}}$ be as in \eqref{eq:fermion}. Then for any $\mathbf{k}$ a pairwise distinct set of indices such that $\vert\mathbf{k}\vert=n$,
\begin{equation}\label{eq:emf}
\partial_s f_{s,\bm{\lambda}}^{\mathrm{Fer}}(\mathbf{k})=2\sum_{i=1}^n\sum_{\substack{\ell\in[\![1,N]\!]\\\ell\notin\{k_1,\dots,k_n\}}}
\frac{f_{s,\bm{\lambda}}^{\mathrm{Fer}}(\mathbf{k}^i(\ell))-f_{s,\bm{\lambda}}^{\mathrm{Fer}}(\mathbf{k})}{N(\lambda_{k_i}-\lambda_\ell)^2}.
\end{equation}
\end{theorem}
Note that these equations actually only need for $\mathbf{u}$ to solve \eqref{eq:dysonvect} and we could consider any deterministic path of eigenvalues (regular enough) or initial condition. Thus \eqref{eq:emf} can also be understood as $\binom{N}{n}$ coupled deterministic (ordinary) differential equations.

	{
	\centering
	Multi-particle representation of \eqref{eq:emf}:\hspace{4em}
	\raisebox{-.1\height}
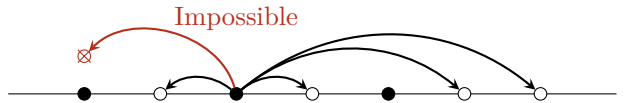
{\begin{tikzpicture}
	\begin{scope}[yshift=-1cm]
		\draw[-] (0,0) to (8,0);
		\foreach \a in {1,3,5}{		
		\node[draw, circle, fill=black, scale=.5] (\a) at (\a,0) {};
		}
		\node[draw, cross, BrickRed, scale=4] (11) at (1,.5) {};
		\node[draw, circle, BrickRed, scale=.5] (11) at (1,.5) {};
		\node[draw, fill=White, circle, scale=.5] (12) at (2,0) {};
		\node[draw, fill=White, circle, scale=.5] (14) at (4,0) {};
		\node[draw, fill=White, circle, scale=.5] (15) at (6,0) {};
		\node[draw, fill=White, circle, scale=.5] (13) at (7,0) {};
		\path[->, >=stealth, thick, BrickRed] (3) edge[bend right=60] (11);
		\node[draw=none] at (3,1) {\cor{Impossible}};
		\path[->, >=stealth, thick] (3) edge[bend right=40] (12);
		\path[->, >=stealth, thick] (3) edge[bend left=40] (13);
		\path[->, >=stealth, thick] (3) edge[bend left=40] (14);
		\path[->, >=stealth, thick] (3) edge[bend left=40] (15);

	\end{scope}
	\end{tikzpicture}}
	
}
\vspace{4ex}
This dynamics on eigenvector moments can also be represented as a multi-particle random walk in a random environment: every configuration of particles has at most one particle on each site and each particle jumps on an empty site at a rate depending on the eigenvalue process.

We call this observable Fermionic by comparison with the observable from \cites{bourgade2017eigenvector, bourgade2018random} which we define now. 
Consider a configuration of $n$ particles $\bm{\xi}:\unn{1}{N}\to \mathbb{N}$  where $\xi_i$ is seen as the number of particles at site $i$. While this formalism makes it simpler to read, we can also write this configuration as an ordered list of the sites where there are particles counted with their multiplicity. In other words, if one denotes $k_1<\dots<k_p$ the indices such that $\xi_{k_i}\geqslant 1$ and $\xi_\ell=0$ for $\ell\notin\{k_1,\dots,k_p\}$, we can write $\bm{\xi}=(k_1,\dots,k_1,k_2,\dots,k_2,\dots,k_p,\dots,k_p)$ where each $k_i$ appears $\xi_{k_i}$ times.

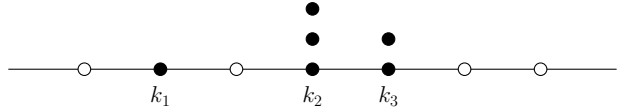
\begin{figure}[H]
	\centering
	\tikzset{every node/.style={draw,circle,fill=black, scale=.5}}
	\begin{tikzpicture}
		\draw[-] (-1,0) -- (7,0);
		\node[fill=White] at (0,0) {};  
 		\node[label={below: \LARGE $k_1$}] at (1,0) {};
		\node[fill=White] at (2,0) {};
		\node[label=below: \LARGE $k_2$] at (3,0) {};
		\node at (3,0.4) {};
		\node at (3,.8) {}; 
		\node[label=below: \LARGE $k_3$] at (4,0) {};
		\node at (4,.4) {};
		\node[fill=White] at (5,0) {};
		\node[fill=White] at (6,0) {};
	\end{tikzpicture}
	\caption{A configuration of particles $\bm{\xi}=(k_1,k_2,k_2,k_2,k_3,k_3)$}
\end{figure}

The Bosonic observable can be given in a matricial way. Given a configuration $\bm{\xi}$ and denoting the particles' positions $(k_1,\dots,k_p)$, define the matrices, 
\[
E^{(ij)}=\left(\frac{\delta_{ik}\delta_{j\ell}+\delta_{i\ell}\delta_{jk}}{1+\delta_{ij}}\right)_{1\leqslant k,\ell\leqslant n}
\quad\text{and}
\quad
Q^{(ij)}_s=\begin{pmatrix}
1 &1\\
1 &1
\end{pmatrix}p_{k_ik_j}(s),
\quad
i,j\in\unn{1}{n}.
\]
Define the following symmetric $2n\times 2n$ matrix involving fluctuations
\[
Q_s(k_1,\dots,k_n)=\sum_{1\leqslant i\leqslant j\leqslant n}E^{(ij)}\otimes Q^{(ij)}_s,\quad\text{for instance}\quad 
Q_s(k_1, k_2)
=
\begin{pmatrix}
p_{k_1k_1}&p_{k_1k_1}&p_{k_1k_2}&p_{k_1k_2}\\
p_{k_1k_1}&p_{k_1k_1}&p_{k_1k_2}&p_{k_1k_2}\\
p_{k_2k_1}&p_{k_2k_1}&p_{k_2k_2}&p_{k_2k_2}\\
p_{k_2k_1}&p_{k_2k_1}&p_{k_2k_2}&p_{k_2k_2}
\end{pmatrix}.
\]
Then the Bosonic observable is 
\begin{equation}\label{eq:perfobs}
f^{\mathrm{Bos}}_{s,\bm{\lambda}}(\bm{\xi})=\frac{1}{\mathcal{M}(\bm{\xi})}\mathds{E}\left[\mathrm{Haf}\,Q_s(\mathbf{\bm{\xi}})
\middle\vert\bm{\lambda}
\right]
\quad\text{with}\quad
\mathcal{M}(\bm{\xi}) \coloneqq \prod_{i=1}^N(2\xi_i)!! = \prod_{i=1}^N \prod_{\substack{k\leqslant 2\xi_i\\k\text{ odd}}}k
\end{equation}
where the Hafnian of a $2n\times 2n$ matrix $A$ is
\[
\mathrm{Haf}\, A = \frac{1}{n!2^n}\sum_{\sigma\in\mathfrak{S}_{2n}}\prod_{j=1}^n A_{\sigma(2j-1),\sigma(2j)}.
\]
The formula \eqref{eq:perfobs} is different than the original one from \cite{bourgade2018random}. It was defined as a sum over perfect matchings on a graph given by the configuration of particles which exactly consists in computing a Hafnian. The following theorem gives the evolution of $f_s^{\mathrm{Bos}}$ which we call the \emph{Bosonic} eigenvector moment flow.

\begin{theorem}[\cite{bourgade2018random}] Suppose that $\bm{u}^s$ is the solution of the Dyson vector flow \eqref{eq:dysonvect} and $f_{s}^{\mathrm{Bos}}(\bm{\xi})$ is given by \eqref{eq:perfobs}.  Then it satisfies the equation
\begin{equation}\label{eq:discs2}
\partial_s f^{\mathrm{Bos}}_{s,\bm{\lambda}}(\bm{\xi})=\sum_{k\neq \ell}\frac{2\xi_k(1+2\xi_\ell)\left(f^{\mathrm{Bos}}_{s,\bm{\lambda}}(\bm{\xi}^{k,\ell})-f^{\mathrm{Bos}}_{s,\bm{\lambda}}(\bm{\xi})\right)}{(\lambda_k(s)-\lambda_\ell(s))^2},
\end{equation}
where $\bm{\xi}^{k,\ell}$ is the configuration obtained by moving a particle from the site $k$ to the site $\ell$.
\end{theorem}

\subsection{Main analytical results: decorrelation and optimal size of overlaps}
The flow of eigenvectors of the Dyson Brownian motion can now be studied through the above two families of observables. Combining the information we can extract through this analysis, we are able to compute new eigenvector statistics. We give here the results for generalized Wigner matrices for simplicity but the results should hold for a wide class of random matrices.
\begin{definition}
Let $W$ be a $N\times N$ symmetric matrix such that its entries $(w_{ij})_{1\leqslant i\leqslant j\leqslant N}$ are centered independent random variables of variance $s_{ij}$ such that there exists two positive constants $c$ and $C$ such that 
\[
\frac{c}{N}\leqslant s_{ij}\leqslant \frac{C}{N}\quad\text{for all }i,j\quad\text{and}\quad \sum_{i,j=1}^Ns_{ij}=1\quad\text{for all }j.
\]
We also assume that the matrix entries have all finite moments in the following sense, for every $p\in\mathbb{N}$ there exists a constant $\mu_p$ independent of $N$ such that
\[
\mathds{E}\left[
	\left(
		s_{ij}^{-1/2}w_{ij}
	\right)^p
\right]\leqslant \mu_p.
\]
\end{definition}

%
 The local behavior of eigenvectors was first considered in the case of Wigner matrices in \cites{knowles2013eigenvector, Tao2012random} with a moment matching condition. It was shown that if two matrix ensembles have the same first four moments, the bulk and edge eigenvectors have asymptotically the same distribution (note that only two moments need to match for edge eigenvectors from \cite{knowles2013eigenvector}).
 
The moment condition was removed in \cite{bourgade2017eigenvector} using a dynamical proof to show asymptotic Gaussianity of projections of eigenvectors:
If $(u_1,\dots,u_N)$ denotes the $\ell^2$-normalized eigenvectors of $W$ a generalized Wigner matrix, for any deterministic set of indices $I\subset\unn{1}{N}$ of fixed cardinality $\vert I\vert=m$ and any
 $k\in\unn{1}{N}$,
\begin{equation}\label{theo:gaussianity}
\left(
	\sqrt{N}u_k(\alpha)
\right)_{\alpha\in I}\xrightarrow[N\rightarrow\infty]{}\left(\mathcal{N}_i\right)_{i=1}^m
\end{equation}
with $(\mathcal{N}_i)$ a family of independent centered unit variance Gaussian random variables and the convergence holds in the sense of moments.

The main contribution of this paper does not concern the Gaussianity of fluctuations of eigenvectors but the correlations between fluctuations. 
\begin{theorem}\label{theo:mainresult}
Let $\vartheta\in(0,\frac{1}{2}]$ be a (small) positive constant. Consider $(k_N,\ell_N)$ two deterministic sequences of indices in $\unn{1}{N}$ such that $k_n\neq \ell_N$, let $I$ be a $N$-dependent set of indices such that $N^{\vartheta}\leqslant \vert I\vert\leqslant N^{1-\vartheta}$ then there exist $\delta_1, \delta_2>0$ depending only on $\vartheta$ such that 
\begin{equation}\label{eq:correlres}
\mathds{E}\left[
	\frac{N^2}{2\vert I\vert}
	\left(
		\sum_{\alpha\in I}u_k(\alpha)^2-\frac{\vert I\vert}{N}
	\right)
	\left(
		\sum_{\alpha\in I}u_\ell(\alpha)^2-\frac{\vert I\vert}{N}
	\right)
\right]\leqslant N^{-\delta_1}.
\end{equation}
Besides, we also have that
\begin{equation}\label{eq:varres}
\left\vert
	\mathds{E}\left[
		\left(
			\frac{N}{\sqrt{\vert I\vert}}\sum_{\alpha\in I}u_k(\alpha)u_\ell(\alpha)
		\right)^2
	\right]
	-
	1
\right\vert
\leqslant N^{-\delta_2}.
\end{equation}
In particular, the last bound gives that for any $\lambda>0$, we have
\begin{equation}
\mathds{P}\left(
	\left\vert
		\sum_{\alpha\in I}
		u_k(\alpha)u_\ell(\alpha)
	\right\vert
	\geqslant
	\lambda
	\frac{\sqrt{\vert I\vert}}{N}
\right)\leqslant \lambda^{-2}(1+o(1)).\label{eq:partialoverlap}
\end{equation}
\end{theorem}
\begin{remark}
The condition on the cardinality of the set $I$ is optimal as correlations between entries of eigenvectors begin to appear as $\vert I\vert \asymp N$ since eigenvectors are normalized and orthogonal. Note also that while the corresponding size of fluctuations of the overlap in \eqref{eq:partialoverlap} is optimal, the probability bound is not.
\end{remark}
To understand the link between Theorem \ref{theo:mainresult} and Theorem \ref{theo:emf}, we can consider both the Bosonic and Fermionic observables with two particles. Consider a configuration $\bm{\xi}$ consisting in one particle at site $k$ and one at site $\ell$. The Bosonic observable from \cite{bourgade2018random} can be written as 
\[
f^{\mathrm{Bos}}_{s,\bm{\lambda}}(\bm{\xi}) = \EL{p_{kk}p_{\ell\ell}+2p_{k\ell}^2}.
\]
While we can analyze this quantity, we cannot extract information on either $\E{p_{kk}p_{\ell\ell}}$ or $\E{p_{k\ell}^2}$. However, if we now consider the corresponding Fermionic observable of the form
\[
f^{\mathrm{Fer}}_{s,\bm{\lambda}}(\bm{\xi}) = \EL{p_{kk}p_{\ell\ell}-p_{k\ell}^2},
\]
we can combine it with the Bosonic observable to obtain Theorem \ref{theo:mainresult}.
\subsection{Organization of the paper}
Section \ref{sec:grassmann} is devoted to the proof of Theorem \ref{theo:emf} using Grassmann variables and the Wick theorem. A combinatorial proof is also given in Appendix \ref{app:proof}. The Fermionic observable can be seen as an anti-commutative Gaussian integral defined in the next section.  

In Section \ref{sec:apriori}, we give several a priori estimates we need for the dynamics analysis, such as local laws, quantum unique ergodicity, and decorrelations of eigenvectors and the resolvent along the dynamics. The local laws were established in previous papers \cites{erdos2012rigidity, bloemendal2014isotropic} and the quantum unique ergodicity has to be developed for edge eigenvectors by adapting the proof from \cite{bourgade2018random}. \\ 
A key new analytic input of our paper is the following. In the analysis of the Fermionic observable, we have to bound terms of the form $\mathds{E}[u_k^s(\alpha) u_k^s(\beta)G^s_{\alpha\beta}(z)]$ where $G(z)=(H_s-z)^{-1}$ is the resolvent. While estimates on the size of eigenvector and resolvent entries are available, such bounds are not sufficient in our paper: the decorrelation between these two quantities needs to be seen. The proof of this decorrelation is based on the stochastic equation followed by the resolvent which we can solve using a method of characteristics. The solution is then given by the initial resolvent at time $s=0$ which, by conditioning on the initial generalized Wigner matrix, decorrelates from the eigenvector $u_k^s$ for $s\gg N^{-1}$.  

These a priori estimates are used in Section \ref{sec:relaxation} to analyze the Fermionic observable first for 2 particles in order to prove Theorem \ref{theo:mainresult} and then for $n$ particles in Theorem \ref{theo:resultinterm} under an additional assumption on the set of indices: $\vert I\vert\ll \sqrt{N}$.
Finally, while we state all our result in the symmetric case, they still hold for Hermitian matrices. While the dynamics \eqref{eq:dysonvect} changes, the Fermionic observable is the same. This is explained in Appendix \ref{app:hermi}.

\paragraph{Acknowledgments.} The author would like to kindly thank his advisors P. Bourgade and S. Péché for their help during the writing of the paper, N. Cook for interesting discussion and the derivation of \eqref{eq:cook} and anonymous referees for helpful suggestions on how to improve the present article.

\section{Proof of Theorem \ref{theo:emf}}\label{sec:grassmann}
In this section, we give a proof of Theorem \ref{theo:emf} using a representation of this determinant as an anticommutative Gaussian integral. For a more through introduction to Grassmann integration we refer the reader to \cite{zinn1989quantum}*{Chapter 1}. 
\subsection{Preliminaries}
The proof of Theorem \ref{theo:emf} involves a supersymmetric representations of our determinant \eqref{eq:fermion}. In order to develop the proof and the tools, we recall in this subsection notions of Grassmann variables and Gaussians expectations with respect to these variables. Grassmann variables can be seen as anticommutative numbers, we first consider four families of Grassmann variables $\{\eta_i,\xi_i, \varphi_i, \psi_i\}_{i=1}^{N}$, they follow the relations of commutation for $i,j$ two indices in $[\![1,N]\!]$ given by
\[
\eta_i\eta_j=-\eta_j\eta_i,\quad \xi_i\xi_j=-\xi_j\xi_i\text{ and}\quad \eta_i\xi_j=-\xi_j\eta_i
\]
and we have similar relations between $(\bm{\eta},\bm{\varphi})$, $(\bm{\eta},\bm{\psi})$, $(\bm{\xi},\bm{\varphi})$, $(\bm{\xi},\bm{\psi}),$ and $(\bm{\varphi},\bm{\psi})$. In particular, see that $\eta_i^2=\xi_i^2=\varphi^2_i=\psi^2_i=0$ and that the variables $\{\eta_i\xi_j\}$ and $\{\varphi_i\psi_j\}$ all commute.
\begin{remark}
A possible representation of such variables is given by matrices. For instance, the Clifford-Wigner-Jordan representation of these Grassmann variables is given by the following: if we want a family of $m$ Grassmann variables $\theta_1,\dots,\theta_m$, one can represent them by $m$ matrices of size $2^m\times 2^m$ with
\[
\theta_i = \bigotimes_{j=1}^{i-1}
\begin{pmatrix}
1 & 0 \\ 0 & -1
\end{pmatrix}
\otimes 
\begin{pmatrix}
0 & 0 \\ 1 & 0
\end{pmatrix}
\otimes
\bigotimes_{j=i+1}^m
\begin{pmatrix}
1 & 0\\0 & 1
\end{pmatrix}.
\]
For instance, we would consider $2^{4N}\times 2^{4N}$ matrices in our case. 
\end{remark} 
Now that we have defined these Grassmann variables, we define our generalized projections. Namely, we can define for a $N$-dimensional vector $v\in\mathbb{C}^N$ the following quantity
\begin{equation}\label{eq:projection}
\langle v\rangle_{\bm{\eta}}=\sum_{\alpha=1}^{N} v(\alpha)\eta_\alpha.
\end{equation}
We can also define functions of these Grassmann variables, note that by Taylor expansion and the commutations rules, it is enough to define polynomials of such variables. Thus we define a function
\[
F(\bm{\eta},\bm{\xi},\bm{\varphi},\bm{\psi})
=
\hspace{-1em}
\sum_{I,J,K,L\subset[\![1,N]\!]}
a_{I,J,K,L}
\prod_{i_1\in I\cap J}\eta_{i_1}\xi_{i_1}
\prod_{j_1\in I\setminus J}\eta_{j_1}
\prod_{k_1\in J\setminus I}\xi_{k_1}
\prod_{i_2\in K\cap L}\varphi_{i_2}\psi_{i_2}
\prod_{j_2\in K\setminus L}\varphi_{j_2}
\prod_{k_2\in L\setminus K}
\psi_{k_2}
\]
where $a_{I,J,K,L}$ are real numbers for our purpose. By the matricial representation, one can then see such a function as a matrix. From this definition of a function, we can define the integral of a function by, 
\[
\int F(\bm{\eta},\bm{\xi},\bm{\varphi},\bm{\psi})\prod_{i=1}^{N} \D\eta_i \D\xi_i \D \varphi_i\D\psi_i= a_{[N],[N],[N],[N]}
\]
where we shortened $[N]:=[\![1,N]\!]$. As explained earlier, we can define functions through a Taylor expansion, in order to construct a Gaussian expectation, we need to construct the exponential. It is straightforward to define it as
\[
\exp\left(
	F(\bm{\eta},\bm{\xi},\bm{\varphi},\bm{\psi})
\right)=\sum_{m=1}^\infty \frac{F(\bm{\eta},\bm{\xi},\bm{\varphi},\bm{\psi})^m}{m!}=\sum_{m=1}^{m_0}\frac{F(\bm{\eta},\bm{\xi},\bm{\varphi},\bm{\psi})^m}{m!}
\]
for some $m_0$ via the commutation relations. We can define our Gaussian expectation as, for an invertible $N\times N$ matrix $\Delta,$
\begin{equation}\label{eq:superexpect}
\mathcal{E}_{\bm{\eta},\bm{\xi},\bm{\varphi},\bm{\psi}}^{\Delta}\left[F(\bm{\eta},\bm{\xi},\bm{\varphi},\bm{\psi})\right]
=
\frac{1}{\det(\Delta^{-1})}\int F(\bm{\eta},\bm{\xi},\bm{\varphi},\bm{\psi})\exp\left(
	\sum_{i,j=1}^{N}\eta_i(\Delta^{-1})_{ij} \xi_j
	+
	\sum_{i=1}^N \varphi_i\psi_i\right)\prod_{i=1}^{2N}\D\eta_i\D\xi_i\D \varphi_i\D \psi_i.
\end{equation}
The Fermionic Wick theorem allows us to compute joint Gaussian moments with respect to this super-expectation. We give it here with respect to our Gaussian expectation and the moments we need later.
\begin{lemma}[Fermionic Wick theorem \cite{zinn1989quantum}*{Section 1.9}]\label{lem:wick}
Consider $\{(i_k,j_k)\}_{k=1}^m\subset [\![1,N]\!]\times[\![1,N]\!]$, and $\{\eta_i,\xi_i, \varphi_i,\psi_i\}_{i=1}^{N}$ a family of Grassmann variables, we have for any $C_0>0$,
\begin{equation}
\mathcal{E}_{\bm{\eta},\bm{\xi},\bm{\varphi},\bm{\psi}}^{\Delta}\left[
	\prod_{k=1}^m
	\left(
		\eta_{i_k}+\I\sqrt{C_0}\varphi_{i_k}
	\right)
	\left(
		\xi_{j_k}+\I\sqrt{C_0}\psi_{i_k}
	\right)
\right]
=
\det \left(
	(\Delta-C_0\mathrm{Id}_m)_{i_k,j_\ell}
\right)_{k,\ell=1}^m.
\end{equation}
\end{lemma}
\subsection{Construction of the Fermionic observable}
In the following definitions we fix a set of indices $I \subset \unn{1}{N}$ and consider $(\mathbf{q}_i)_{i\in I}$ a family of vectors of $\mathbb{R}^N$ not necesarily orthogonal. We are able to construct an observable based on the families of Grassmann variables which follows \eqref{eq:emf}. Then by taking the Gaussian expectation defined in \eqref{eq:superexpect} we obtain the observable \eqref{eq:fermion} by choosing the right covariance matrix $\Delta$ depending on the family $(\mathbf{q}_i)$.

We consider the observable, for $\mathbf{u}^s$ the solution to the Dyson vector flow \eqref{eq:dysonvect}
\begin{equation}\label{eq:defferm}
g^{\mathrm{Fer}}_s(k_1,\dots,k_n)=g^{\mathrm{Fer}}_s(k_1,\dots,k_n;\bm{\eta},\bm{\xi},\bm{\varphi},\bm{\psi})\coloneqq\mathds{E}\left[
	\prod_{i=1}^n\langle {u}^s_{k_i}\rangle_{\bm{\eta}+\I\sqrt{C_0}\bm{\varphi}}
	\langle {u}^s_{k_i}\rangle_{\bm{\xi}+\I\sqrt{C_0}\bm{\psi}}
\middle\vert \bm{\lambda}
\right].
\end{equation}
Note that we dropped the subscript $\bm{\lambda}$ for simplicity and will do so for $f^{\mathrm{Fer}}_s$ and $f^{\mathrm{Bos}}_s$ in the rest of the article.
\begin{remark}
Note that in this definition, the product is commutative since we have quantities of order 2 in Grassmann variables. See also that this is a similar quantity as the moment observable from \cite{bourgade2017eigenvector}. Indeed if one considers a configuration with a single particle at sites $k_1,\dots,k_n$ then the observable would be written as 
\[
g^{\mathrm{Bos}}_s(k_1,\dots,k_n)=\mathds{E}\left[
	\prod_{i=1}^n\scp{\mathbf{q}}{u_{k_i}}^2
	\middle\vert\bm{\lambda}
\right].
\]
\end{remark}
In order to see that $g_s^{\mathrm{Fer}}$ follows a form of the eigenvector moment flow \eqref{eq:emf}, first see the following proposition from \cite{bourgade2017eigenvector} which gives us the generator of the Dyson vector flow.
\begin{proposition}[\cite{bourgade2017eigenvector}]\label{prop:generator}
The generator acting on smooth functions of the diffusion \eqref{eq:dysonvect} is given by 
\begin{equation}
L_t = \sum_{1\leqslant k<\ell\leqslant N}\frac{1}{N(\lambda_k-\lambda_\ell)^2}X_{k\ell}^2
\end{equation}
with the operator $X_{k\ell}$ defined by
\begin{equation}
X_{k\ell}=\sum_{\alpha=1}^N\left(
	u_k(\alpha)\partial_{u_\ell(\alpha)}-u_\ell(\alpha)\partial_{u_k(\alpha)}
\right)
\end{equation}
\end{proposition}

We thus need to prove the following lemma, showing that $g_s^{\mathrm{Fer}}$ follows the eigenvector moment flow

\begin{lemma}\label{lem:emf} For $g_s^{\mathrm{Fer}}$ defined as in \eqref{eq:defferm} and $\mathbf{k}=(k_1,\dots,k_n)$ with $k_i\neq k_j$ for $i\neq j$, we have
\begin{equation}
\partial_s g_s^{\mathrm{Fer}}(\mathbf{k})=\sum_{i=1}^n\sum_{\substack{\ell=1\\\ell\notin\{k_1,\dots,k_n\}}}^N\frac{g_s^{\mathrm{Fer}}(\mathbf{k}^i(\ell))-g_s^{\mathrm{Fer}}(\mathbf{k})}{N(\lambda_{k_{i}}-\lambda_\ell)^2}.
\end{equation}
\end{lemma}
\begin{proof}
First see by definition of the operator that since the eigenvectors for $k\notin\{k_1,\dots,k_n\}$ are not considered in the observable $g_s^{\mathrm{Fer}}(\mathbf{k})$ we clearly have
\[
X_{k\ell}^2g_s^{\mathrm{Fer}}(\mathbf{k})=0\quad\text{for }k\notin\{k_1,\dots,k_n\}.
\]
Now, we need to show that for fixed $i,j\in\unn{1}{n}$ we also have $X_{k_i k_j}^2g_s^{\mathrm{Fer}}(\mathbf{k})=0$. This equality actually comes from the anticommutativity of the Grassmann variables. First see that we have the relations
\[
X_{k\ell}\langle {u}_k\rangle_{\bm{\eta}}=-\langle {u}_\ell\rangle_{\bm{\eta}}\quad\text{and}\quad X_{k\ell}\langle {u}_\ell\rangle_{\bm{\eta}}=\langle {u}_k\rangle_{\bm{\eta}}.
\]
Besides, by definition of the operator $X_{k\ell}$ we only need to look at the part of the observable involving the eigenvectors ${u}_{k}$ and ${u}_{\ell}$, hence computing the quantity
\begin{align*}
X^2_{k_i k_j}&\left(
	\langle {u}_{k_i}\rangle_{\bm{\eta}}\langle {u}_{k_i}\rangle_{\bm{\xi}}\langle {u}_{k_j}\rangle_{\bm{\eta}}\langle {u}_{k_j}\rangle_{\bm{\xi}}
\right)
=
2(
	\langle {u}_{k_j}\rangle_{\bm{\eta}}\langle {u}_{k_j}\rangle_{\bm{\xi}}\langle {u}_{k_j}\rangle_{\bm{\eta}}\langle {u}_{k_j}\rangle_{\bm{\xi}}
	+
	\langle {u}_{k_i}\rangle_{\bm{\eta}}\langle {u}_{k_i}\rangle_{\bm{\xi}}\langle {u}_{k_i}\rangle_{\bm{\eta}}\langle {u}_{k_i}\rangle_{\bm{\xi}}
	\\
	&-
	2\langle {u}_{k_i}\rangle_{\bm{\eta}}\langle {u}_{k_i}\rangle_{\bm{\xi}}\langle {u}_{k_j}\rangle_{\bm{\eta}}\langle {u}_{k_j}\rangle_{\bm{\xi}}
	-
	\langle {u}_{k_j}\rangle_{\bm{\eta}}\langle {u}_{k_i}\rangle_{\bm{\xi}}\langle {u}_{k_i}\rangle_{\bm{\eta}}\langle {u}_{k_j}\rangle_{\bm{\xi}}
	-
	\langle {u}_{k_i}\rangle_{\bm{\eta}}\langle {u}_{k_j}\rangle_{\bm{\xi}}\langle {u}_{k_j}\rangle_{\bm{\eta}}\langle {u}_{k_i}\rangle_{\bm{\xi}}\\
	&-
	\langle {u}_{k_j}\rangle_{\bm{\eta}}\langle {u}_{k_i}\rangle_{\bm{\xi}}\langle {u}_{k_j}\rangle_{\bm{\eta}}\langle {u}_{k_i}\rangle_{\bm{\xi}}
	-
	\langle {u}_{k_i}\rangle_{\bm{\eta}}\langle {u}_{k_j}\rangle_{\bm{\xi}}\langle {u}_{k_i}\rangle_{\bm{\eta}}\langle {u}_{k_j}\rangle_{\bm{\xi}}
)
\\
&=0
\end{align*}
where we used the fact that $\langle {u}_{k_i}\rangle_{\bm{\eta}}^2=0$ and the anticommutativity relations. Finally, we need to compute $X_{k_i\ell}^2g^{\mathrm{Fer}}_s(\mathbf{k})$ for $i\in\unn{1}{n}$ and $\ell\in\unn{1}{N}\setminus\{k_1,\dots,k_N\},$ to do so we just need to compute
\[
X_{k_i\ell}^2\langle {u}_{k_i}\rangle_{\bm{\eta}}\langle {u}_{k_i}\rangle_{\bm{\xi}}=2\left(
	\langle {u}_{\ell}\rangle_{\bm{\eta}}\langle {u}_{\ell}\rangle_{\bm{\xi}}-
	\langle {u}_{k_i}\rangle_{\bm{\eta}}\langle {u}_{k_i}\rangle_{\bm{\xi}}
\right)
\]
which means that we have
\[
X^2_{k_i\ell}g_s^{\mathrm{Fer}}(\mathbf{k})=2\left(
	g_s^{\mathrm{Fer}}(\mathbf{k}^i(\ell))-g_s^{\mathrm{Fer}}(\mathbf{k})
\right).
\]
Combining all these equalities, we obtain Lemma \ref{lem:emf}.
\end{proof}

We now only need to show that we can obtain $f_s^{\mathrm{Fer}}$ using our observable $g^{\mathrm{Fer}}_s$, this involves the Fermionic Wick theorem given by Lemma \ref{lem:wick}.
\begin{lemma}
There exists $\Delta$ such that 
\[
\mathcal{E}^\Delta_{\bm{\eta},\bm{\xi},\bm{\varphi},\bm{\psi}}\left[
	g_s^{\mathrm{Fer}}(\mathbf{k})
\right]=f_s^{\mathrm{Fer}}(\mathbf{k}).
\]
\end{lemma}
\begin{proof}
By definition of $g_s^{\mathrm{Fer}}$, we have the following, forgetting the dependence in $s$,
\begin{multline*}
\mathcal{E}^\Delta_{\bm{\eta},\bm{\xi},\bm{\varphi},\bm{\psi}}\left[
	\prod_{i=1}^n\langle{u}_{k_i}\rangle_{\bm{\eta}+\I\sqrt{\vert C_0\vert}\bm{\varphi}}
	\langle {u}_{k_i}\rangle_{\bm{\xi}+\I\sqrt{\vert C_0\vert}\bm{\psi}}
\right]\\
=
\sum_{\substack{i_1,\dots,i_n\\j_1,\dots,j_n}}^{N}
\mathcal{E}^\Delta_{\bm{\eta},\bm{\xi},\bm{\varphi},\bm{\psi}}
\left[
	\prod_{i=1}^n
	\left(
		\eta_{i_k}+\I\sqrt{\vert C_0\vert}\varphi_{i_k}
	\right)
	\left(
		\xi_{j_k}+\I\sqrt{\vert C_0\vert}\psi_{j_k}	
	\right)
\right]	
\prod_{m=1}^n{u}_{k_m}(i_m){u}_{k_m}(j_m).
\end{multline*}
Now we can use the Fermionic Wick theorem \ref{lem:wick} in order to compute these Gaussian moments,
\[
\mathcal{E}^\Delta_{\bm{\eta},\bm{\xi},\bm{\varphi},\bm{\psi}}\left[
	g_s^{\mathrm{Fer}}(\mathbf{k})
\right]
=
\sum_{\substack{i_1,\dots,i_n\\j_1\dots j_n}}^N
\det \left(
	\left(
		\Delta-C_0\mathrm{Id}
	\right)_{i_pj_q}
	u_{k_p}(i_p)u_{k_q}(j_q)
\right)_{p,q=1}^n.
\]
Thus by multilinearity of the determinant we obtain that
\begin{equation*}
\mathcal{E}^\Delta_{\bm{\eta},\bm{\xi},\bm{\varphi},\bm{\psi}}\left[
	g_s^{\mathrm{Fer}}(\mathbf{k})
\right]
=
\det\left(
	\sum_{i,j=1}^{N}
	\left(
		\Delta-C_0\mathrm{Id}
	\right)_{ij}{u}_{k_p}(i){u}_{k_q}(j)
\right)_{p,q=1}^n
\hspace{-2em}=
\det\left(
	\sum_{i,j=1}^{N}
		\Delta_{ij}{u}_{k_p}(i){u}_{k_q}(j)
		-
		C_0\mathds{1}_{k_p=k_q}
\right)_{p,q=1}^n.
\end{equation*}
Now, we consider the following covariance matrix
\[
\Delta_{ij}=\sum_{\alpha\in I}q_\alpha(i)q_\alpha(j)\quad\text{for }i,j\in\unn{1}{N}.
\]
Thus we can finally see that the entries of the matrix we take the determinant of are given by, for $\alpha,\beta\in\unn{1}{n},$
\[
\sum_{i,j=1}^{N}\Delta_{ij}{u}_{k_\alpha}(i){u}_{k_\beta}(j)
-C_0\mathds{1}_{k_\alpha=k_\beta}=\sum_{i\in I}\scp{\mathbf{q}_i}{u_{k_\alpha}}\scp{\mathbf{q}_i}{u_{k_\beta}}-C_0\mathds{1}_{k_\alpha=k_\beta}=p_{k_\alpha k_\beta}(s).
\]
\end{proof}

We constructed a Gaussian integral to represent our Fermionic observable, given by a determinant, as a Gaussian moment. The same construction can be done for the Bosonic observable. Indeed, the Hafnian from \eqref{eq:perfobs} can also be represented by a Gaussian moment \cite{cook}. We explain here the construction of the Bosonic observable from the original moment observable from \cite{bourgade2017eigenvector}.
\[
\text{Let}\quad
\mathbf{q}=\mathbf{q}^{(1)}+\I \sqrt{\frac{\vert I\vert}{N}}\mathbf{q}^{(2)}
\quad\text{with}\quad
\mathbf{q}^{(1)}_{\alpha}=\mathcal{N}_{\alpha}\mathds{1}_{\alpha\in I}
\quad\text{and}\quad
\mathbf{q}^{(2)}_{\alpha}
=
\mathcal{N}_{\alpha}^\prime
\quad\text{for}\quad
\alpha\in\unn{1}{N}
\] 
where $(\mathcal{N}_\alpha)_{\alpha\in I}$ and $(\mathcal{N}^\prime_{\alpha})_{\alpha=1}^N$ are two independent families of independent centered Gaussian with unit variance.
We then have the following identity, where $\mathds{E}_{\mathbf{q}}$ denotes the expectation with respect to the two families of Gaussian random variables
\begin{equation}\label{eq:cook}
f^{\mathrm{Bos}}_s(\bm{\xi})
=
\frac{1}{\mathcal{M}(\bm{\xi})}
\mathds{E}\left[
	\mathds{E}_{\mathbf{q}}\left[
		\prod_{i=1}^N
		\scp{\mathbf{q}}{u_k^s}^{2\xi_i}
		\right]
	\middle\vert
	\bm{\lambda}	
\right].
\end{equation}
Indeed, the Hafnian appears when using Wick's rule. We give the construction in the simplest case where the fluctuations are given by
\[
p_{kk}(s)=\sum_{\alpha\in I}u_k^s(\alpha)^2-\frac{\vert I\vert}{N}
\quad\text{and}\quad
p_{k\ell}(s)=\sum_{\alpha\in I}u_k^s(\alpha)u_\ell^s(\alpha)
\] 
but the same generalization by changing the centering and adding correlation between the Gaussian random variables can be executed to obtain the observable in its most general form.  From this construction, it is clear that $f^{\mathrm{Bos}}_s$ follows the eigenvector moment flow as it is a direct consequence of \cite{bourgade2017eigenvector}. 
\begin{remark}
We gave here a proof of Theorem \ref{theo:emf} with supersymmetry and a link to the first observable following this equation from \cite{bourgade2017eigenvector}. However, knowing Proposition \ref{prop:generator}, it is possible to give a combinatorial proof of the theorem with no consideration of Grassmann variables but simply of the properties of the determinant. We give this proof in Appendix \ref{app:proof}.
\end{remark}

\section{A priori estimates on the dynamics}\label{sec:apriori}
In this section, we derive or recall some a priori estimates on eigenvalues or eigenvectors along the dynamics. The proof of Theorem \ref{theo:mainresult} is based upon the three-step strategy used to prove universality of eigenvalues and eigenvectors of random matrices first introduced in \cites{erdos2010bulk, erdos2011universality} (see \cites{erdos2017dynamical} for recent book on the subject). The first step of the strategy is a local law. The second step consists on a short time relaxation by the Dyson Brownian motion and is developed in Section \ref{sec:relaxation}. Finally, the last step corresponds to comparison between our model and the dynamics at a small time $s$. The next subsection is dedicated to local laws for our model.

\subsection{Local laws}

A local law consists of a high-probability bound on the resolvent of our generalized Wigner matrix controlling it down to the optimal scale $N^{-1+\varepsilon}$ for any $\varepsilon>0$.

Define the resolvent $G$ and the Stieltjes transform of the semicircle law $m$ to be for $z\in\mathbb{C}$ with $\Im z>0$
\begin{equation}\label{eq:resolvent}
G(z)=\sum_{k=1}^N\frac{\vert u_k\rangle \langle u_k\vert}{\lambda_k-z}\quad\text{and}\quad m(z)=\int\frac{\D \rho_{\mathrm{sc}}(x)}{x-z}=\frac{-z+\sqrt{z^2-4}}{2}
\quad\text{with}\quad
\rho_{\mathrm{sc}}(\D x)
=
\frac{\mathds{1}_{[-2,2]}(x)}{2\pi}\sqrt{4-x^2}\D x
\end{equation}
where the choice of the square root is given by $m$ being holomorphic in the upper half plane and $m(z)\rightarrow 0$ as $z\rightarrow \infty$.
We need two forms of local law, one is an averaged local law on the Stieltjes transform of the empirical spectral distribution of $W$, $s(z)=N^{-1}\Tr G(z),$ the other is on the resolvent as a quadratic form, also called an isotropic local law.
\begin{theorem}[\cites{erdos2012rigidity, bloemendal2014isotropic}]
Consider the following spectral domain, for any (small) $\omega>0,$
\[
\mathcal{D}_{\omega}=
\left\{
	z=E+\I \eta,\,\vert E\vert\leqslant \omega^{-1},\,N^{-1+\omega}\leqslant \eta\leqslant \omega^{-1}
\right\},
\]
then we have for any positive $\varepsilon$ and $D>0$,
\begin{equation}\label{eq:avelocal}
\sup_{z\in\mathcal{D}_\omega}\mathds{P}
\left(
	\left\vert
		s(z)-m(z)
	\right\vert\geqslant \frac{N^\varepsilon}{N\eta}
\right)\leqslant N^{-D},
\end{equation}
and for any vector $\mathbf{v},$ $\mathbf{w} \in\mathbb{R}^N$, for any positive $\varepsilon$ and $D$,
\begin{equation}\label{eq:isolocal}
\sup_{z\in\mathcal{D}_\omega}\mathds{P}\left(
	\left\vert
		\scp{\mathbf{v}}{G(z)\mathbf{w}}-m(z)\scp{\mathbf{v}}{\mathbf{w}}
	\right\vert
	\geqslant
	N^\varepsilon
	\Vert \mathbf{v}\Vert \Vert\mathbf{w}\Vert
	\left(
		\sqrt{\frac{\Im m(z)}{N\eta}}+\frac{1}{N\eta}
	\right)
\right)\leqslant N^{-D}.
\end{equation}
\end{theorem}
As a corollary of this theorem, one obtains the complete delocalization of eigenvectors as an overwhelming probability bound. We need this optimal estimate (up to logarithmic corrections) in order to control eigenvectors.
\begin{corollary}\label{cor:deloc}
Let $k\in\unn{1}{N}$ and $\mathbf{q}\in\mathbb{R}^N$ such that $\Vert \mathbf{q}\Vert_2=1$, we have, for any $D$ and any $\varepsilon$ positive
\begin{equation}\label{eq:deloc}
\mathds{P}\left(
	\left\vert
		\scp{\mathbf{q}}{u_k}
	\right\vert
	\geqslant
	\frac{N^\varepsilon}{\sqrt{N}}
\right)\leqslant N^{-D}.
\end{equation}
\end{corollary}
Another corollary of the optimal bound \ref{eq:avelocal} is a rigidity estimate on the eigenvalues. It states that eigenvalues of our generalized Wigner matrix are close to their deterministic classical locations. These locations, denoted $(\gamma_k)$ are defined in the following implicit way:
\[
\int_{-\infty}^{\gamma_k} \D \rho_{\mathrm{sc}}(x)
=
\frac{k}{N}.
\]
\begin{theorem}[\cite{erdos2012rigidity}]
If we denote $\hat{k}=\min(k,N+1-k)$, for any $\varepsilon>0$ and any $D>0$ we have
\begin{equation}
\mathds{P}\left(
	\text{there exists }k\in\unn{1}{N},\,
	\left\vert 
		\lambda_k-\gamma_k
	\right\vert
	\geqslant
	\hat{k}^{-1/3}N^{-2/3+\varepsilon}
\right)
\leqslant
N^{-D}.\label{eq:rigidity}
\end{equation}
\end{theorem}

The estimates \eqref{eq:avelocal}, \eqref{eq:isolocal}, \eqref{eq:deloc} and \eqref{eq:rigidity} hold along the dynamics \eqref{eq:dyson}. This is the statement of the following lemma.

\begin{lemma}[\cite{bourgade2017eigenvector}*{Lemma 4.2}]\label{lem:goodset}
Let $\xi, \omega>0,$.
Consider $W$ a generalized Wigner matrix and consider the dynamics \eqref{eq:dyson} $(H_s)_{0\leqslant s\leqslant 1}$ with $H_0=W$. Define the resolvent and its normalized trace for $z$ in the upper plane,
\[
G^s(z) = (H_s-z)^{-1}\quad\text{and}\quad m_s(z)=\frac{1}{N}\Tr G^s(z).
\] 
It induces a measure on the space of eigenvalues and eigenvectors $(\bm{\lambda}(s),\bm{u}^s)$ for $0\leqslant s\leqslant 1$ such that the following event $\mathcal{A}_1(\xi)$ holds with overwhelming probability in the sense that for any $D>0$,
\[
\mathds{P}\left(
	\mathcal{A}_1(\xi)
\right)\geqslant 1-N^{-D}
\]
\begin{itemize}
\item We have rigidity of eigenvalues: 
\[
\text{For all } s\in[0,1], \,\vert \lambda_k(s)-\gamma_k\vert \leqslant N^{-2/3+\xi}(\hat{k})^{-1/3} \text{ uniformly in }k\in[\![1,N]\!].
\]
\item The averaged local law hold: for all $s\in[0,1]$, uniformly in $z=E+\I\eta\in\mathcal{D}_\omega$,
\(
\vert m_s(z)-m(z)\vert\leqslant \frac{N^\xi}{N\eta}
\)
\item When we condition on the trajectory $(\bm{\lambda}(s))_{s\in[0,1]}\in \mathcal{A}_1$, the entrywise local law holds: for all $s\in[0,1]$, uniformly in $z=E+\I \eta\in\mathcal{D}_\omega$ and $\alpha,\beta\in\unn{1}{N}$,
\[
\left\vert
	G^s_{\alpha\beta}(z)-m(z)\delta_{\alpha\beta}
\right\vert
\leqslant N^\xi
\left(
	\sqrt{\frac{\Im m(z)}{N\eta}}+\frac{1}{N\eta}
\right),
\]
and eigenvector delocalization holds: $\forall s\in[0,1],$ $\Vert {u_k^s}\Vert_\infty^2\leqslant N^{-1+\xi}$ uniformly in $k\in[\![1,N]\!].$
\end{itemize} 
\end{lemma}
This lemma allows us to prove most results deterministically by working on the event $\mathcal{A}_1$ which holds with overwhelming probability. Note that we keep only the dependence in $\xi$ in the definition of $\mathcal{A}_1$ but it also depends on the choice of $\omega$.
\subsection{Quantum unique ergodicity}
In this subsection, we state a priori results we need on the $p_{k\ell}(s)$. This overwhelming probability bound for $p_{k\ell}$ was studied for Gaussian divisible ensembles in \cite{bourgade2018random} in order to study band matrices but only consider bulk eigenvectors. While they only consider $\vert I\vert\geqslant cN$ for some constant $c>0$, we adapt the proof to any $\vert I\vert$ and we obtain the following result in the case of generalized Wigner matrices.

From now on, we consider the case where
\begin{equation}\label{eq:defpkk2}
p_{kk}(s)=\sum_{\alpha\in I}u_k^s(\alpha)^2-\frac{\vert I\vert}{N}
\quad\text{and}\quad
p_{k\ell}(s)= \sum_{\alpha\in I}u_k^s(\alpha)u_\ell^s(\alpha).
\end{equation}
\begin{proposition}\label{lem:queedge}
First denote the following error parameter,
\[
\Psi_1(s) = \frac{\vert I\vert}{N^{3/2}s^2}+\sqrt{\frac{\vert I\vert}{N^2s^{3}}}.
\] 
Let $\omega>0$, we have for any $\varepsilon$ and $D$ positive and $s\in[N^{-1/3+\omega},1]$,
\[
\mathds{P}\left(
\sup_{k,\ell\in\unn{1}{N}}
\vert p_{kk}(s)\vert
+
\vert p_{k\ell}(s) \vert
\geqslant
N^\varepsilon\Psi_1(s)
\right)\leqslant N^{-D}.
\]
\end{proposition}
\begin{remark}
The error term is the sum of two terms and it is not clear whether one is larger than the other since it depends on the regime of $\vert I\vert$ or $s$. 
\end{remark}
Note that in the case of bulk eigenvectors, we have the following overwhelming probability bound from \cite{bourgade2018random} for a general class of initial condition. 
\begin{theorem}[\cite{bourgade2018random}*{Theorem 2.5}]\label{theo:que}
Let $\alpha\in (0,1)$ and a small $\omega>0$, for $k,\ell\in \unn{\alpha N}{(1-\alpha)N}$ (indices in the bulk) and $\vert I\vert \geqslant cN$ for some $c>0$, we have that for any $\varepsilon,D>0$ and $N^{-1+\omega}\leqslant s \leqslant 1$
\[
\mathds{P}\left(
	\vert p_{kk}(s)\vert
	+
	\vert p_{k\ell}(s)\vert
	\geqslant
	\frac{N^\varepsilon}{\sqrt{Ns}}
\right)\leqslant N^{-D}.
\]
\end{theorem}
Before beginning the proof of Proposition \ref{lem:queedge}, we need the following lemma relating our fluctuations $p_{k\ell}$ to the Bosonic observable.

\begin{lemma}[\cite{bourgade2018random}]\label{lem:holder}
Take an even integer $n$, there exists a $C>0$ depending on $n$ such that for any $i<j$ and any time $s$ we have
\begin{equation}
\mathds{E}\left[p_{ij}(s)^n\middle\vert\bm{\lambda}\right]\leqslant C\left(f_s^{\mathrm{Bos}}(\bm{\xi}^{(1)})+f_s^{\mathrm{Bos}}(\bm{\xi}^{(2)})+f_s^{\mathrm{Bos}}(\bm{\xi}^{(3)})\right)
\end{equation}
where $\bm{\xi}^{(1)}$ is the configuration of $n$ particles in the site $i$ and no particle elsewhere, $\bm{\xi^}{(2)}$ $n$ particles in the site $j$, and $\bm{\xi}^{(3)}$ an equal number of particles between the site $i$ and the site $j$. 
\end{lemma}

Using this lemma we can now adapt the proof of \cite{bourgade2018random}*{Theorem 2.5} to the edge case. Note that the proof is actually simpler since we do not need to localize the dynamics in the bulk of the spectrum.
\begin{proof}[Proof of Proposition \ref{lem:queedge}]
Let $\xi>0$, we first condition on $H_0$ and $(\bm{\lambda}(s))_{s\in[0,1]}$ belonging to $\mathcal{A}_1(\xi).$ This allows to work deterministically on the eigenvalue path and our initial symmetric matrix. Let $\omega>0$ be such that $3\omega/2>\xi$ and $s_0,s_1\in [N^{-1/3+\omega},1]$ such that $s_0\leqslant s_1$.

 Consider $f^{\mathrm{Bos}}(\bm{\xi})$ the Bosonic obervable for the eigenvector moment flow. Consider $n$ fixed and look at a configuration $\bm{\xi_m}$ to be such that 
\[
f^{\mathrm{Bos}}_s(\bm{\xi_m})=\sup_{\bm{\xi},\,\mathcal{N}(\bm{\xi})=n} f^{\mathrm{Bos}}_s(\bm{\xi})\quad \text{and}\quad S_{s_0,s_1}=\sup_{s\in[s_0,s_1]}\sup_{\bm{\xi},\,\mathcal{N}(\bm{\xi})=n} f_s^{\mathrm{Bos}}(\bm{\xi}).
\]
Note that if there are several maximizers, we pick one aribtrarily under the constraint that $\bm{\xi_m}$ remains piecewise constant in $s$. Let $\eta=s_0^2N^{-\omega}\geqslant N^{-2/3+\omega}$. We then have, forgetting about the superscript Bos, for all $s\in[s_0,s_1]$,
\[
\partial_s f_s(\bm{\xi_m})=\sum_{k\neq \ell}2\eta_k(1+2\eta_\ell)\frac{f_s(\bm{\xi_m^{k,\ell}})-f_s(\bm{\xi_m})}{N(\lambda_k-\lambda_\ell)^2}
\leqslant
\frac{C}{N\eta}\sum_{i=1}^p
\sum_{\ell\neq k_i}
\frac{\eta(f_s(\bm{\xi_m^{k,\ell}})-f_s(\bm{\xi_m}))}{(\lambda_{k_i}-\lambda_\ell)^2+\eta^2}
\] 
where we denoted $(k_1,\dots,k_p)$ the sites $k$ such that $\eta_k\neq 0$. In particular, $p\leqslant n$ and $\sum_{i=1}^p\eta_{k_i}=n$. Now, we have that 
\[
f_s(\bm{\xi_m})\frac{1}{N}\sum_{i=1}^p\sum_{\ell \neq k_i}\frac{\eta}{(\lambda_{k_i}-\lambda_{\ell})^2+\eta^2}
=
\left(
	\sum_{i=1}^p
	\Im m(z_{k_i})
\right)
f_s(\bm{\xi_m})
+
\mathcal{O}\left(
	\frac{N^\xi}{N\eta}S_{s_0,s_1}
\right)
\]
where we denoted $z_{k_i}=\lambda_{k_i}+\I \eta$. For the other term, we use an implicit bound using H{ö}lder inequalities. First, we can remove some terms in the sum, 
\[
\Im \sum_{\ell\neq k_i}\frac{f_s(\bm{\xi_m^{k,\ell}})}{N(\lambda_\ell-z_{k_i})}=\Im \sum_{\ell\notin\{k_1,\dots,k_p\}}
\frac{f_s(\bm{\xi_m^{k,\ell}})}{N(\lambda_\ell-z_{k_i})}
+\mathcal{O}{}\left(\frac{N^\xi}{N\eta}S_{s_0,s_1}\right).
\]
Now, we can expand by the definition of $f_s(\bm{\xi})$ in terms of a sum over perfect matchings. Since we move one particle from $k$ to $\ell$, which is an empty site for the configuration $\bm{\xi_m}$, we only have two particles in the graph on the site $\ell$. Thus, there is two possibilities for the perfect matching, either there is an edge $\{(\ell,1),(\ell,2)\}$ or there is not. If there is such an edge, then we can write the contribution of such perfect matchings as
\[
\mathds{E}\left[
	Q_{n-1}(\bm{\xi_m})\Im \sum_{\ell\notin\{k_1,\dots,k_p\}}\frac{p_{\ell\ell}}{N(\lambda_\ell-z)}
	\middle\vert
	\bm{\lambda}
\right].
\]
Now, from the definition of $\mathcal{A}_1$ we have that 
\[
\Im \sum_{\ell\notin\{k_1,\dots,k_p\}}
\frac{p_{\ell\ell}}{N(\lambda_\ell-z)}
=
\Im \frac{1}{N}\sum_{\alpha\in I}G^s_{\alpha\alpha}(z) - \frac{\vert I\vert}{N}\Im m(z)
=
\mathcal{O}{}\left(
	\frac{N^\xi\vert I\vert}{N\sqrt{N\eta}}
\right).\]
See that $Q_{n-1}(\bm{\xi})$ is a sum of monomial of degree $n-1$ involving the fluctuations $p_{k\ell}$. Thus by a Young inequality, using Lemma \ref{lem:holder}, we have that 
\[
Q_{n-1}(\bm{\xi})
=
\mathcal{O}{}\left(S_{s_0,s_1}^{\frac{n-1}{n}}\right).
\]
Now for perfect matchings where $\{(\ell,1),(\ell,2)\}$ is not an edge, we can write the contribution in the following way,
\[
\mathds{E}\left[
	Q_{n-2}(q_1,q_2,\bm{\xi_m})\Im \sum_{\ell\notin\{k_1,\dots,k_p\}}\frac{p_{k_{q_1}\ell}p_{k_{q_2}\ell}}{N(\lambda_\ell-z)}
	\middle\vert
	\bm{\lambda}
\right].
\]
We can write, in order to control the sum
\[
\Im \sum_{\ell\notin\{k_1,\dots,k_p\}}
\frac{p_{k_{q_1}\ell}p_{k_{q_2}\ell}}{N(\lambda_\ell-z)}
=\mathcal{O}{}\left(
\frac{1}{N\eta}
\sum_{\ell=1}^N 
(p_{k_{q_1}\ell}^2+p_{k_{q_2}\ell}^2)
\right)=\mathcal{O}{}\left(
	\frac{N^\xi\vert I\vert}{N^2\eta}
\right)
\]
where we used the delocalization property from the definition of $\mathcal{A}_1$ and the fact that 
\[
\sum_{\ell=1}^N p_{k\ell}^2=\sum_{\alpha\in I}u_k(\alpha)^2=\mathcal{O}{}\left(
	N^\xi\frac{\vert I\vert}{N}
\right).
\]
In the same way, using a Young inequality with Lemma \ref{lem:holder}, we control the polynomial of degree $n-2$ in terms of $p_{k\ell}$,
\[
Q_{n-2}(q_1,q_2,\bm{\xi_m})=\mathcal{O}\left(
	S_{s_0,s_1}^{\frac{n-2}{n}}
\right).
\]
Thus, combining all these inequalities, we obtain the following Gronwall-type inequality,
\[
\partial_sf_s(\bm{\xi_m})
\leqslant 
-\frac{C}{\eta}\left(
\sum_{i=1}^p\Im m(z_{k_i})
\right)
f_s(\bm{\xi_m})
+
\mathcal{O}{}\left(
	\frac{N^\xi}{\eta}
	\left(
		\frac{1}{N\eta}S_{s_0,s_1}
		+
		\frac{\vert I\vert}{N\sqrt{N\eta}}S_{s_0,s_1}^{\frac{n-1}{n}}
		+
		\frac{\vert I\vert}{N^2\eta}S_{s_0,s_1}^{\frac{n-2}{n}}
	\right)
\right).
\]
Now, using the fact that since $\eta\geqslant N^{-2/3+\omega}$, we have $\Im m(E+\I \eta)\geqslant \sqrt{\eta}$, we have the bound
\[
\partial_s f_s(\bm{\xi_m})
\leqslant 
-\frac{C}{\sqrt{\eta}}f_s(\bm{\xi_m})
+
\mathcal{O}\left(
	\frac{N^\xi}{\sqrt{\eta}}
	\left(
		\frac{1}{N\eta^{3/2}}S_{s_0,s_1}
		+
		\frac{\vert I\vert}{N^{3/2}\eta}S_{s_0,s_1}^{\frac{n-1}{n}}
		+
		\frac{\vert I\vert}{N^2\eta^{3/2}}S_{s_0,s_1}^{\frac{n-2}{n}}
	\right)
\right).
\]
By Gronwall's lemma, since $s\geqslant s_0 \gg \sqrt{\eta}$,
\[
f_s(\bm{\xi_m})
=
\mathcal{O}{}\left(
	\frac{N^{3\omega/2+\xi}}{Ns_0^3}S_{s_0,s_1}+\frac{\vert I\vert N^{\xi+\omega}}{N^{3/2}s_0^2}S_{s_0,s_1}^{\frac{n-1}{n}}
	+
	\frac{\vert I\vert N^{\xi+3\omega/2}}{N^2s_0^3}S_{s_0,s_1}^{\frac{n-2}{n}}
	+
	N^{-D}
\right)
\]
for any $D>0$ and for all $s_0\leqslant s \leqslant \frac{s_0+s_1}{2}$ for instance. Thus we obtain that 
\begin{equation}\label{eq:induct}
S_{s_0,s_2}\leqslant C\frac{N^{3\omega/2+\xi}}{Ns_0^3}S_{s_0,s_1}+C\frac{\vert I\vert N^{\xi+\omega}}{N^{3/2}s_0^2}S_{s_0,s_1}^{\frac{n-1}{n}}
+
C\frac{\vert I\vert N^{\xi+3\omega/2}}{N^2s_0^3}S_{s_0,s_1}^{\frac{n-2}{n}}
+
CN^{-D}
\end{equation}
with $s_2=\frac{s_0+s_1}{2}\in [s_0,s_1]$. Since $s_0\geqslant N^{-1/3+\omega}$ we obtain that 
\[
 C\frac{N^{3\omega/2+\xi}}{Ns_0^3}S_{s_0,s_1}
 \leqslant CN^{-3\omega/2+\xi}S_{s_0,s_1}=CN^{-\sigma}S_{s_0,s_1}\quad\text{with}\quad \sigma\coloneqq \frac{3\omega}{2}-\xi>0.
\]
We can induct the bound \eqref{eq:induct} by defining a sequence of time $(s_i)$ such that $s_i = \frac{s_0+s_{i-1}}{2}$ then we have that as long as 
\[
S_{s_0,s_i}^{1/n}
\geqslant 
N^{\sigma}\frac{\vert I\vert N^{\xi+\omega}}{N^{3/2}s_0^2}
+
N^{\sigma/2}\sqrt{\frac{\vert I\vert N^{\xi+3\omega/2}}{N^2s_0^3}}
\]
the following bound holds
\[
S_{s_0,s_{i+1}}
\leqslant N^{-i\sigma}S_{s_0,s_1}.
\]
If this bound holds indefinitely then our result is proved. Otherwise, it means that for some $i_0$ we have that
\[
S_{s_0,s_i}\leqslant
\left(
	N^{\sigma}\frac{\vert I\vert N^{\xi+\omega}}{N^{3/2}s_0^2}
	+
	N^{\sigma/2}\sqrt{\frac{\vert I\vert N^{\xi+3\omega/2}}{N^2s_0^3}}
\right)^n	
\]
We can finish by a Markov inequality, since $\omega$ and $\xi$ (and thus $\sigma$) can be taken arbitrarily small in the proof, for any $\varepsilon, D>0$ we can consider $n=n(\varepsilon,D)$ sufficiently large such that 
\[
\mathds{P}\left(
\vert p_{kk}(s_0)\vert
+
\vert p_{k\ell}(s_0) \vert
\geqslant
N^\varepsilon\Psi_1(s_0)
\right)\leqslant N^{n\varepsilon/2}\Psi_1(s_0)^n(N^\varepsilon \Psi_1(s_0))^{-n}+N^{-D}\leqslant CN^{-D}
\]
where we used Corollary \ref{lem:holder}, and Lemma \ref{lem:goodset}.
\end{proof}
\subsection{Decorrelation of eigenvectors and the resolvent}
An estimate we need in Section \ref{sec:relaxation} is a correlation between eigenvector entries and the resolvent. Of course, the resolvent defined in \eqref{eq:resolvent} clearly depends on all eigenvectors of the random matrix. However, it can be seen as an average over eigenvectors and do not depend significantly on a single eigenvector. We prove such a decorrelation estimate dynamically and obtain the result after some time $s$.
\begin{proposition}\label{prop:decorr}
Let $\xi,\,\omega,\delta^\prime$ be small positive constants and $j,\,\alpha,\,\beta\in\unn{1}{N},$ with $\alpha\neq \beta$. For a relaxation time $s\in[N^{-2/3+\delta^\prime}, N^{-\delta^\prime}]$ and for $\delta_1>0$ small enough such that $N^{-\delta_1}s\gg N^{-2/3}$, for any $z=E+\I\eta\in\mathcal{D}_\omega$ and any $D>0$, we have
\[
\begin{gathered}
\mathds{P}\left(
	\left\vert
		\EL{u_j^s(\alpha)u_j^s(\beta)\Im G_{\alpha\beta}^s(z)}
	\right\vert
	\geqslant {N^{5\xi+\delta_1}\Psi_2(s,\eta)}
\right)\leqslant N^{-D}\\
\text{with}\quad
\Psi_2(s,\eta)
=
\frac{1}{N^2\eta}
+
\frac{1}{N^2s^{3/4}\eta^{1/2}}
+
\frac{\sqrt{s}}{N^2\eta^2}
\end{gathered}
\] 
uniformly in $j,\,\alpha,\,\beta$ and $z$ where the expectation is taken over $(B_{k\ell}(t))_{t\in[0,1]}$ and the probability is taken over the initial matrix $H_0$ and  $(B_{k\ell}(t))_{t\in[0,1]}$.
\end{proposition}
\begin{remark}
Note that this proposition gives us a better bound than the trivial bound one can do by simply using the delocalization \eqref{eq:deloc} and the local law \eqref{eq:isolocal} if we consider a relaxation time $s$ close to order 1 such as $s=N^{-\theta}$ for a small $\theta$. Indeed, we would obtain a bound of the form
\[
\EL{u_j^s(\alpha)u_j^s(\beta)\Im G_{\alpha\beta}^s(z)}
=
\O{\frac{N^{\xi^\prime}}{N^{3/2}s^{1/2}}}
\] 
and lose an order of $\sqrt{N}$.
\end{remark}
We prove this proposition by considering the dynamics of the resolvent $G^s$ and use the characteristics method to allow decorrelations with the eigenvectors. The characteristics method express the resolvent at time $s$ by the the initial resolvent $G^0$ which permits decorrelation via a correct conditioning on the initial condition $H_0$. A similar dynamics was used in \cite{bourgade2018extreme} to study extreme gaps between eigenvalues of generalized Wigner matrices. The observable \cite{bourgade2018extreme}*{(1.11)} looks similar to a resolvent but the dependence in eigenvectors in the numerator in \eqref{eq:resolvent} is instead replaced by an eigenvalue coupling observable. Interestingly, the dynamics of this eigenvalue coupling observable is similar to the one followed by eigenvector entries, the main difference coming from the martingale term in the stochastic differential equation which involves here off-diagonal entries of the Dyson Brownian motion.  The dynamics of the resolvent $G^s$ is given in the following lemma.
\begin{lemma}
Consider $(\bm{\lambda}(s), \mathbf{u}^s)$ the solution to \eqref{eq:dysonval} and \eqref{eq:dysonvect}, for any $z$ such that $\Im z\neq 0$, if one defines
\[
\widetilde{G}^s=\mathrm{e}^{-s/2}G^s,
\]
we have for any $\alpha,\,\beta\in\unn{1}{N}$,
\begin{equation}\label{eq:dynareso}
\D \widetilde{G}_{\alpha\beta}^s(z)
=
\left(
	s(z)+\frac{z}{2}
\right)
 \partial_z \widetilde{G}^s_{\alpha\beta}(z)\D s
-
\frac{1}{\sqrt{N}}
\sum_{k,\ell=1}^N
\frac{u_k(\alpha)u_\ell(\beta)\D B_{k\ell}}{(\lambda_k-z)(\lambda_\ell-z)}.
\end{equation}
\end{lemma}
Note that while we state it for the dynamics with a generalized Wigner matrix as an initial condition. This stays true for any arbitrary initial condition.
\begin{proof}
By using Itô's formula, we have
\begin{equation*}
\D u_k(\alpha)u_k(\beta)=
\sum_{\ell\neq k}
\frac{u_\ell(\alpha)u_\ell(\beta)-u_k(\alpha)u_k(\beta)}{N(\lambda_k-\lambda_\ell)^2}\D s
+
\frac{1}{\sqrt{N}}
\sum_{\ell\neq k}
\frac{\D B_{k\ell}}{\lambda_k-\lambda_\ell}(u_k(\alpha)u_\ell(\beta)+u_k(\beta)u_\ell(\alpha)).
\end{equation*}
And then we have 
\begin{align*}
\D \sum_{k=1}^N\frac{u_k(\alpha)u_k(\beta)}{\lambda_k-z}
=&
\frac{1}{\sqrt{N}}\sum_{k\neq\ell}\frac{u_k(\alpha)u_\ell(\beta)+u_k(\beta) u_\ell(\alpha)}{(\lambda_k-\lambda_\ell)(\lambda_k-z)}\D B_{k\ell}
-
\frac{1}{\sqrt{N}}\sum_{k=1}^N\frac{\D B_{kk}}{(\lambda_k-z)^2}u_k(\alpha)u_k(\beta)
\\
&+
\sum_{k\neq \ell}
\frac{u_\ell(\alpha)u_\ell(\beta)-u_k(\alpha)u_k(\beta)}{N(\lambda_k-\lambda_\ell)^2(\lambda_k-z)}\D s
-\sum_{k\neq \ell}
\frac{u_k(\alpha)u_k(\beta)}{N(\lambda_k-\lambda_\ell)(\lambda_k-z)^2}\D s
\\&
+
\sum_{k=1}^N\frac{\lambda_ku_k(\alpha)u_k(\beta)}{2(\lambda_k-z)^2}\D s
+
\sum_{k=1}^N\frac{u_k(\alpha)u_k(\beta)}{N(\lambda_k-z)^3}\D s.
\end{align*}
We can use the identity
\[
\frac{u_k(\alpha)u_k(\beta)}{(\lambda_k-\lambda_\ell)}
\left(
	\frac{1}{(\lambda_k-\lambda_\ell)(\lambda_\ell-z)}
	-
	\frac{1}{(\lambda_k-\lambda_\ell)(\lambda_k-z)}
	-
	\frac{1}{(\lambda_k-z)^2}
\right)
=
\frac{u_k(\alpha)u_k(\beta)}{(\lambda_k-z)^2(\lambda_\ell-z)}
\]
to obtain
\begin{align}
\D \sum_{k=1}^N
\frac{u_k(\alpha)u_k(\beta)}{\lambda_k-z}
=
\sum_{k,\ell=1}^N\frac{u_k(\alpha)u_k(\beta)}{N(\lambda_k-z)^2(\lambda_\ell-z)}\D s
-
\frac{1}{2\sqrt{N}}
\sum_{k,\ell=1}^N
\frac{u_k(\alpha)u_\ell(\beta)+u_k(\beta)u_\ell(\alpha)}{(\lambda_k-z)(\lambda_\ell-z)}
\D B_{k\ell}.
\end{align}
By th spectral decomposition of the resolvent from \eqref{eq:resolvent} and the definition of $\widetilde{G}^s$, we obtain \eqref{eq:dynareso}.
\end{proof}
Equation \eqref{eq:dynareso} can be seen as a stochastic advection equation. If one removes the stochastic martingale term and replace $s(z)$ by its deterministic equivalent $m(z)$, it is possible to solve this equation using the characheristics method.
\begin{lemma}
Let the characteristic $z_s$ be defined as
\[
z_s=\frac{1}{2}\left(
	\mathrm{e}^{s/2}(z+\sqrt{z^2-4})
	+
	\mathrm{e}^{-s/2}(z-\sqrt{z^2-4})
\right)
\]
we then have for any function $h_0$ smooth enough
\[
\partial_s h_s(z)
=
\left(
	m(z)+\frac{z}{2}
\right)
\partial_z h_s(z)
\quad\text{with}\quad
h_s(z) := h_0(z_s)
\]
where $m(z)$ is given in \eqref{eq:resolvent}.
\end{lemma}
We are now ready to prove Proposition \ref{prop:decorr} using this advection equation and the previous representation of its solution.
\begin{proof}[Proof of Proposition \ref{prop:decorr}]
Firstly, see that we condition on $H_0$ and a path of eigenvalues $\bm{\lambda}$ in $\mathcal{A}_1(\xi)$ which holds with overwhelming probability. Now, we have
\[
\widetilde{G}^s_{\alpha\beta}(z)-\widetilde{G}^0_{\alpha\beta}(z_s)
=
-\int_0^s \D \widetilde{G}^{s-\tau}_{\alpha\beta}(z_\tau)
=
\frac{1}{\sqrt{N}}\sum_{k,\ell=1}^N\int_0^s
\frac{u_{k}^\tau(\alpha)u_{\ell}^\tau(\beta)\D B_{k\ell}(\tau)}{(\lambda_k(\tau)-z_{s-\tau})(\lambda_\ell(\tau)-z_{s-\tau})}
+\O{\frac{N^\xi}{N\eta}}.
\]
We have the easy bound, since $s\ll 1$,
\begin{equation*}
c
\EL{u_{j}^s(\alpha)u_{j}^s(\beta)\Im \widetilde{G}^s_{\alpha\beta}(z)}
\leqslant
C\mathds{E}\left[
	u_{j}^s(\alpha)u_{j}^s(\beta)
	\Im G_{\alpha\beta}^s(z)
	\middle\vert
	\bm{\lambda}
\right]
\leqslant C
\EL{u_{j}^s(\alpha)u_{j}^s(\beta)\Im \widetilde{G}^s_{\alpha\beta}(z)}
\end{equation*}
which gives 
\begin{multline}\label{eq:decomp}
\mathds{E}\left[
	u_{j}^s(\alpha)u_{j}^s(\beta)
	\Im G_{\alpha\beta}^s(z)
	\middle\vert
	\bm{\lambda}
\right]
\leqslant
\EL{u_{j}^s(\alpha)u_{j}^s(\beta)\Im \widetilde{G}^0_{\alpha\beta}(z_s)}\\
+
\frac{1}{\sqrt{N}}
\EL{\Im
\sum_{k,\ell=1}^N
\int_0^s
\frac{u_{j}^s(\alpha)u_{j}^s(\beta)u_{k}^\tau(\alpha)u_{\ell}^\tau(\beta)}{(\lambda_k(\tau)-z_{s-\tau})(\lambda_\ell(\tau)-z_{s-\tau})}\D B_{k\ell}(\tau)
}
+
\O{\frac{N^{2\omega}}{N^2\eta}}.
\end{multline}
For the first term, we use the fact that $G^0(z_s)$ only depends on the initial matrix for the dynamics and we can write
\[
\EL{u_{j}^s(\alpha)u_{j}^s(\beta)\Im \widetilde{G}^0_{\alpha\beta}(z_s)}
=
\EL{\Im \widetilde{G}^0_{\alpha\beta}(z_s)\ELH{u_{j}^s(\alpha)u_{j}^s(\beta)}}
\]
But we have the dynamics for $\tilde{f}_s(j)=\ELH{u_{j}^s(\alpha)u_{j}^s(\beta)}$,
\[
\partial_s \tilde{f}_s(j)
=
\sum_{k\neq j}
\frac{\tilde{f}_s(k)-\tilde{f}_s(j)}{N(\lambda_k-\lambda_j)^2}.
\]
This a clear consequence of Itô's formula and similar to the eigenvector moment flow from \cite{bourgade2017eigenvector}. So that by a maximum principle we obtain the following: define the index ${k_m}$ as
\[
\tilde{f}_s({k_m})
=
\sup_{k\in\unn{1}{N}}\tilde{f}_s({k})
\]
where we choose an arbitrary maximizer if there are several with the condition that $k_m$ is a piecewise constant function of $s$.
then we can write for any $\eta^\prime>0$,
\begin{align*}
\partial_s \tilde{f}_s({k_m})
&\leqslant 
-\frac{1}{\eta^\prime}
\left(
	\tilde{f}_s({k_m})
	\frac{1}{N}\sum_{k\neq {k_m}}
	\frac{\eta^\prime}{(\lambda_k-\lambda_{k_m})^2+{\eta^\prime}^2}
	-\frac{1}{N}\sum_{k\neq {k_m}}
	\frac{\tilde{f}_s(k)\eta^\prime}{(\lambda_k-\lambda_{k_m})^2+{\eta^\prime}^2}
\right)
\\&
\leqslant
-\frac{1}{\eta^\prime}
\left(
	\Im m(\lambda_{{k_m}}+\I\eta^\prime)
	\tilde{f}_s({k_m})
	-
	\Im \sum_{k\neq {k_m}}
	\frac{\tilde{f}_s(k)}{N(\lambda_k-(\lambda_{{k_m}}+\I\eta^\prime))}
	+
	\O{\frac{N^{3\xi}}{N^2\eta^\prime}}
\right)
\end{align*}
where we used the definition of $\mathcal{A}_1(\xi)$ from Lemma \ref{lem:goodset} which gives $\tilde{f}_s\leqslant N^{-1+2\xi}$ to obtain our error term.
The second term in the inequality can be bounded using the entrywise local law of the resolvent \eqref{eq:isolocal}, by denoting $z_{k_m}=\lambda_{k_m}+\I\eta^\prime,$
\begin{equation*}
\Im\sum_{k\neq {k_m}}
\frac{\tilde{f}_s(k)}{N(\lambda_k-z_{k_m})}
=
\frac{1}{N}\ELH{\Im G_{\alpha\beta}^s(z_{k_m})}+\O{\frac{N^{2\xi}}{N^2\eta^\prime}}
=
\O{\frac{N^{2\xi}}{N}\left(\sqrt{\frac{\Im m(z_{k_m})}{N\eta^\prime}}+\frac{1}{N\eta^\prime}\right)}.
\end{equation*}
So that finally
\[
\partial_s \tilde{f}_s(k_m)
\leqslant 
-\frac{1}{\eta^\prime}\left(
	\Im m(z_{k_m})\tilde{f}_s(k_m)
	+
	\O{\frac{N^{3\xi}}{N}\left(\sqrt{\frac{\Im m(z_{k_m})}{N\eta^\prime}}+\frac{1}{N\eta^\prime}\right)}
\right)
\]
which gives, using a Gronwall argument for $\eta^\prime=N^{-\delta_1} s$ where we choose $s$ and $\delta_1$ such that $\eta^\prime\geqslant N^{-2/3},$
\begin{equation}\label{eq:bounddecor1}
\tilde{f}_s(k_m)
=
\O{\frac{N^{3\xi}}{N\sqrt{N\eta^\prime\Im m(z_{k_m})}}}
=
\O{\frac{N^{3\xi+3\delta_1/4}}{N^{3/2}s^{3/4}}}.
\end{equation}
In the last inequality, we used that $\Im m(z_{k_m})$ is bounded below by $\sqrt{\eta^\prime}=N^{-\delta_1/2}s^{1/2}$.
Finally since $\Im z_s$ is increasing along the characteristics we have that $\Im z_s\geqslant \eta$ which gives that using the local law,
\begin{equation}\label{eq:bounddecor2}
\Im \widetilde{G}^0_{\alpha\beta}(z_s)\leqslant \frac{N^\xi}{\sqrt{N\eta}}.
\end{equation}
Finally we have the following bound combining \eqref{eq:bounddecor1} and \eqref{eq:bounddecor2},
\begin{equation}\label{eq:bound2}
\EL{u_{j}^s(\alpha)u_{j}^s(\beta)\Im \widetilde{G}^0_{\alpha\beta}(z_s)}
=
\O{
\frac{N^{4\xi+3\delta_1/4}}{N^2s^{3/4}\eta^{1/2}}}.
\end{equation}

We now need to bound the term with the stochastic integral in \eqref{eq:decomp}. Firstly see that we condition on the whole path of eigenvalues $\bm{\lambda}$ which corresponds to condition on the $\sigma$-field generated by the $B_{kk}$. However, it is independent of the noise driving the eigenvector dynamics \eqref{eq:dysonvect} which gives a zero expectation for all off-diagonal terms in the sum. Thus we simply need to bound
\[
\frac{1}{\sqrt{N}}
\EL{u_{j}^s(\alpha)u_{j}^s(\beta)
M_s
}
\quad\text{with}\quad
M_s = \int_0^s
\sum_{k=1}^N\frac{u_{k}^\tau(\alpha)u_{k}^\tau(\beta)}{(\lambda_k(\tau)-z_{s-\tau})^2}\D B_{kk}(\tau).
\]
We bound the stochastic integral using its quadratic variation via the probability bound, for any $\varepsilon>0$ and $D>0$,
\[
\mathds{P}\left(
	\sup_{0\leqslant u\leqslant s}\vert M_u\vert \geqslant N^\varepsilon \sqrt{\langle M\rangle_s}
\right)
\leqslant N^{-D}.
\]
We have the inequality
\[
\langle M\rangle_s
\leqslant
\int_0^s
\sum_{k=1}^N
\frac{u_{k}^\tau(\alpha)^2u_{k}^\tau(\beta)^2}{\vert\lambda_k(\tau)-z_{s-\tau}\vert^4}\D \tau
\leqslant
\frac{N^{4\xi}}{N^2}
N
\int_0^s \frac{\D \tau}{\eta^4}
\leqslant
\frac{N^{4\xi}s}{N\eta^4}.
\]
In this suboptimal inequality, we used the complete delocalization of eigenvectors to bound the numerator and the fact that $\vert \lambda_k(\tau)-z_{s-\tau}\vert\geqslant \Im z_{s-\tau}\geqslant \eta$. Finally, using the complete delocalization at time $s$ we can now bound
\begin{equation}\label{eq:bound1}
\frac{1}{\sqrt{N}}
\EL{
u_{j}^s(\alpha)u_{j}^s(\beta)
M_s
}
=
\O{\frac{N^{4\xi}\sqrt{s}}{N^2\eta^{2}}}.
\end{equation}
Finally, putting \eqref{eq:decomp}, \eqref{eq:bound2} and \eqref{eq:bound1} together we obtain
\[
\EL{u_{j}^s(\alpha)u_{j}^s(\beta)\Im G_{\alpha\beta}^s(z)}
=
\O{
\frac{N^{3\xi}}{N^2\eta}
+
\frac{N^{4\xi+3\delta_1/4}}{N^2s^{3/4}\eta^{1/2}}
+
\frac{N^{4\xi}\sqrt{s}}{N^2\eta^2}
}.
\]
\end{proof}
To give an idea of the bound from Proposition \ref{prop:decorr}, we give now the parameters scaling we choose when using it. In the next section, we consider a small relaxation time $s=N^{-\theta}$ for a small $\theta\in(0,2/3)$ and the spectral resolution is chosen such that $\eta=N^{-\delta_2}s$. If we take the parameter $\delta_1$ such that $\delta_1\leqslant 2\delta_2$ then the third term in the bound is dominating and we have
\[
\EL{u_{j}^s(\alpha)u_{j}^s(\beta)\Im G_{\alpha\beta}^s(z)}
=
\O{
\frac{N^{4\xi+2\delta_1+3\theta/2}}{N^2}
}.
\]
Since we suppose that $\vert I\vert\leqslant CN^{1-\vartheta}$ for some $\vartheta\in (0,1)$ and choosing our positive parameters $\xi$, $\delta_1$ and $\theta$ small enough such that $\kappa := \vartheta - (4\xi+2\delta_1+3\theta/2)>0$, we obtain that 
\[
\EL{u_{j}^s(\alpha)u_{j}^s(\beta)\Im G_{\alpha\beta}^s(z)}
=
\O{N^{-\kappa}\frac{N^\vartheta}{N^2}}
=
\O{\frac{N^{-\kappa}}{N\vert I\vert}}.
\]
All in all, it can be used to bound the following quantity needed in the proof of the decorrelations of eigenvector fluctuations
\[
\frac{1}{N}\sum_{\substack{\alpha,\beta\in I\\\alpha \neq\beta}}
\EL{u_{j}^s(\alpha)u_{j}^s(\beta)\Im G_{\alpha\beta}^s(z)}
=
\O{N^{-\kappa}\frac{\vert I\vert}{N^2}}
\]
which gives the correct estimates as the leading order is given by the diagonal elements
\[
\frac{1}{N}\sum_{\alpha\in I}u_j^s(\alpha)^2\Im G_{\alpha\alpha}^s(z)
\approx
\frac{\vert I\vert}{N^2}\Im m(z).
\]

\section{Relaxation by the Dyson Brownian motion}
\label{sec:relaxation}
In this section, we make our initial matrix $W$ undergo the Dyson Brownian motion from Definition \ref{def:dyson}. The point being to obtain the asymptotic value of $f^{\mathrm{Fer}}_s$ after a short time $s$ and see that it coincides with the family $(p_{k\ell})_{k\ell}$ being independent Gaussian random variables.

We give the asymptotic value of $f^{\mathrm{Fer}}$ in the following lemma, while this is a simple computation, we give a short proof in order to see a recursion relation. Indeed, it is this recursion relation which occurs in the later proof.
\begin{lemma}\label{lem:detgauss}
Consider $A_n=\mathds{E}\left[
	\det \mathrm{G}
\right]$ where $\mathrm{G}$ is a symmetric $n\times n$ matrix with independent entries (up to the symmetry) given by $\mathrm{G}_{ij}\sim \mathcal{N}(0,1)$ for $i\neq j$ and $\mathrm{G}_{ii}\sim \mathcal{N}(0,2)$. Then we have
\[
A_n=\left\{
	\begin{array}{ll}
		(-1)^{n/2}n!!\quad&\text{if }n\text{ is even},\\
		0&\text{otherwise}.
	\end{array}
\right.
\]
\end{lemma}
\begin{proof}
This can be computed using a recursion relation of order 2 by developing according to some rows and columns. We write in the following $M_{(i)}^{(j)}$ the matrix $M$ where we removed the line $i$ and the column $j$. We can then develop the determinant in the following way
\begin{equation*}
A_n=\mathds{E}\left[
	\sum_{i=1}^n \mathrm{G}_{1,i} (-1)^{i+1}\det \mathrm{G}_{(1)}^{(i)}
\right]
=
\sum_{i=1}^n\sum_{j=1}^{n-1}(-1)^{i+j}\mathds{E}\left[
	\mathrm{G}_{1,i}\mathrm{G}_{j+1,1}
\right]
\mathds{E}\left[
	\det \mathrm{G}_{(1,j+1)}^{(i,1)}
\right]
\\=-\sum_{i=2}^n A_{n-2}.
\end{equation*}
And this recursion formula gives us the result knowing that
\[
A_1=0\quad\text{and}\quad A_2=-1.
\]
\end{proof}

In order to prove Theorem \ref{theo:mainresult} we only need to analyze the Fermionic and Bosonic observable for two particles. This is done in the next subsection and separately from the case of $n$ particles. Indeed, the estimate from Proposition \ref{prop:decorr} cannot be generalized to the case of $n$ particles and only the delocalization with the local law is used. This gives the correct estimate only if the set of indices is of cardinality $\vert I\vert\ll\sqrt{N}$ while we can reach the optimal $\vert I\vert\ll N$ for two particles.

In the rest of the paper we consider our family of vectors ($\mathbf{q}_\alpha)$ to be the canonical basis and $C_0=\vert I\vert/N$ the correct centering. Note that we could generalize in the same way to any orthonormal basis of vector but consider only this special case for simplicity. The constant, however, need to be given by this value to consider fluctuations of a centered random variable. So that we consider the family $(p_{k\ell})$ as in \eqref{eq:defpkk2}.
\subsection{Fermionic observable with two particles}
For two particles, we want to asymptotically compute our Fermionic observable. First, we denote the probability event $\mathcal{A}_2(\xi,\delta_1,\omega)$ the intersection of $\mathcal{A}_1(\xi,\omega)$ with the events where the bounds from Propositions \ref{prop:decorr} and \ref{lem:queedge} hold uniformly in all indices $k,\ell,\alpha\in\unn{1}{N}$ and $z\in\mathcal{D}_\omega$. Note that we need to consider a relaxation of time $s\gg N^{-2/3}$ for this to be possible. The goal of this subsection is to prove the following theorem.
\begin{proposition}
Let $\vartheta\in(0,\frac{1}{2})$ and $I\subset\unn{1}{N}$ such that $N^{\vartheta}\leqslant \vert I\vert \leqslant N^{1-\vartheta}$. We can take $\xi(\vartheta),\delta_1(\vartheta), \omega(\vartheta)>0$ sufficiently small such that if $f_s^{\mathrm{Fer}}$ is as in \eqref{eq:fermion}, conditionally on $H_0,\bm{\lambda}\in\mathcal{A}_2(\xi,\delta_1,\omega)$, there exists $\theta_2=\theta_2(\vartheta,\xi,\delta_1,\omega)$  sufficiently small such that for $N^{-\theta_2}\leqslant s\leqslant 1$ there exists a positive $\vartheta_2=\vartheta_2(s,\xi\,\delta_1,\omega,\vartheta)$ such that
\begin{equation}\label{eq:resultinterm1}
\sup_{k_1\neq k_2}
\left\vert
	f_s^{\mathrm{Fer}}(k_1,k_2)
	+
	\frac{\vert I\vert}{N^2}
\right\vert
=
\O{N^{-\vartheta_2}\frac{\vert I\vert}{N^2}}.
\end{equation}
\end{proposition}
\begin{proof}
For readibility, we omit the time dependence for our overlaps, our eigenvectors and eigenvalues. We can write our Fermionic observable as 
\[
f_s^{\mathrm{Fer}}(k_1,k_2)=\mathds{E}\left[
	p_{k_1k_1}p_{k_2k_2}-p_{k_1k_2}^2
	\middle\vert
	\bm{\lambda}
\right].
\]
We use a maximum principle for this observable since it follows the parabolic equation \eqref{eq:emf} and obtain the result by a Gronwall argument. Consider $\mathbf{k^m}=(k^m_1,k^m_2)$ the multi-index corresponding to the maximum of the function $f_s^{\mathrm{Fer}}$ chosen as in the earlier proofs so that 
\[
f_s^{\mathrm{Fer}}(\mathbf{k^m})=\sup_{\mathbf{k},\,\vert \mathbf{k}\vert=2}f_s^{\mathrm{Fer}}(\mathbf{k}).
\]
Then we have, since $f_s^{\mathrm{Fer}}$ follows \eqref{eq:emf} and $\mathbf{k^m}$ is the index for which $f_s^{\mathrm{Fer}}$ is the maximum, for any positive $\eta$,
\begin{multline*}
\partial_s f_s^{\mathrm{Fer}}(\mathbf{k^m})
=
2\sum_{i=1}^2
\sum_{\ell\notin\{k_1^m,k_2^m\}}\frac{f_s^{\mathrm{Fer}}((\mathbf{k^m})^i(\ell))-f_s^{\mathrm{Fer}}(\mathbf{k^m})}{N(\lambda_{k_i^m}-\lambda_\ell)^2}
\\
\leqslant
\frac{2}{\eta}\sum_{i=1}^2
\frac{1}{N}\!\sum_{\ell\notin\{k_1^m,k_2^m\}}\!\frac{(f_s^{\mathrm{Fer}}((\mathbf{k^m})^i(\ell))-f_s^{\mathrm{Fer}}(\mathbf{k^m}))\eta}{(\lambda_{k_i^m}-\lambda_\ell)^2+\eta^2}.
\end{multline*}

We consider only the terms in the first sum of the right hand side for readability, the other term can be bounded in exactly the same way. First note that adding this parameter $\eta$ made imaginary parts arise. Namely, we have the formula with $z_k\coloneqq \lambda_k+\I\eta$,
\[
\frac{1}{N}\sum_{\ell\notin \{k_1^m,k_2^m\}}f_s^{\mathrm{Fer}}(\mathbf{k^m})\frac{\eta}{(\lambda_{k_i^m}-\lambda_\ell)^2+\eta^2}
=
f_s^{\mathrm{Fer}}(\mathbf{k^m})\Im m(z_{k_i})
+
\mathcal{O}\left(
	\frac{N^\xi}{N\eta}\Psi_1^2(s)
\right)
\]
where we used Proposition \ref{lem:queedge} and \eqref{eq:avelocal} for the error term. Now, we need to control the term involving $f_s^{\mathrm{Fer}}((\mathbf{k^m})^i(\ell))=\mathds{E}\left[p_{k_{3-i}k_{3-i}}p_{\ell\ell}-p_{k_{3-i}\ell}\middle\vert \bm{\lambda}\right].$ As said before we consider the term $i=1$, thus we can write
\begin{equation}\label{eq:interm1}
\frac{1}{N}\sum_{\ell\notin\{k_1^m,k_2^m\}}\frac{(p_{\ell\ell}p_{k_2k_2}-p_{\ell k_2}^2)\eta}{(\lambda_{k_1}-\lambda_\ell)^2+\eta^2}=p_{k_2k_2}\Im \sum_{\ell=1}^N\frac{p_{\ell\ell}}{N(\lambda_\ell-z_{k_1})}-\Im \sum_{\ell=1}^N\frac{p_{\ell k_2}^2}{N(\lambda_{\ell}-z_{k_1})}
+
\mathcal{O}{}\left(
	\frac{N^\xi}{N\eta}\Psi_1^2(s)
\right).
\end{equation}
These two sums can be written in terms of the resolvent defined in \eqref{eq:resolvent} and control them with the isotropic local law \eqref{eq:isolocal}. Indeed, by definition of our overlaps,
\begin{multline}\label{eq:pll}
\Im \sum_{\ell=1}^N\frac{p_{\ell\ell}}{N(\lambda_\ell -z_{k_1})}
=
\Im \sum_{\alpha \in I}\sum_{\ell=1}^N\frac{u_\ell(\alpha)^2}{N(\lambda_\ell-z_{k_1})}-\frac{\vert I\vert}{N}\Im m(z_{k_1})
+\mathcal{O}\left(
	\frac{N^\xi\vert I\vert}{N^2\eta}
\right)
\\
=
\frac{1}{N}\Im\left(
	\sum_{\alpha\in I}\left(
		G_{\alpha\alpha}(z_{k_1})-m(z_{k_1})
	\right)
\right)
=
\mathcal{O}{}\left(
\frac{N^\xi\vert I\vert}{N\sqrt{N\eta}}	
\right).
\end{multline}
For the second term in \eqref{eq:interm1}, we can also write it in terms of the resolvent but we need to keep the conditional expectation on the path of eigenvalues,
\begin{multline*}
\Im \sum_{\ell=1}^N\frac{\EL{p_{\ell k_2}^2}}{N(\lambda_\ell-z_{k_1})}=\frac{1}{N}\sum_{\alpha,\beta\in I}\EL{u_{k_2}(\alpha)u_{k_2}(\beta)\Im G_{\alpha\beta}(z_{k_1})}
\\=\frac{1}{N}\sum_{\alpha\in I}\EL{{u_{k_2}}(\alpha)^2}\Im m(z_{k_1})
+
\mathcal{O}{}\left(
	\frac{\vert I\vert^2}{N}N^{5\xi+\delta_1}\Psi_2(s,\eta)
\right)
\end{multline*}
where we used Proposition \ref{prop:decorr}. To evaluate the first term, we can use Proposition \ref{lem:queedge} which gives us
\begin{equation}\label{eq:plk2}
\Im \sum_{\ell=1}^N\frac{\EL{p_{\ell k_2}^2}}{N(\lambda_\ell-z_{k_1})}
=
\frac{\vert I\vert}{N^2}\Im m(z_{k_1})
+
\O{\frac{N^\xi}{N}\Psi_1(s)+\frac{\vert I\vert^2}{N}N^{5\xi+\delta_1}\Psi_2(s,\eta)}.
\end{equation}

 Now, combining the estimates \eqref{eq:interm1}, \eqref{eq:pll} and \eqref{eq:plk2} we obtain the inequality,
\begin{multline*}
\partial_s\left(
	f^{\mathrm{Fer}}_s(\mathbf{k^m})+\frac{\vert I\vert}{N^2}
\right)
\leqslant
-\frac{\Im m(z_{k_1})+\Im m(z_{k_2})}{\eta}\left(
	f^{\mathrm{Fer}}_s(\mathbf{k^m})+\frac{\vert I\vert}{N^2}
\right)
\\
+
\mathcal{O}{}\left(
	\frac{1}{\eta}\left(
		\frac{N^{\xi}}{N\eta}\Psi_1^2(s)+\frac{N^\xi\vert I\vert}{N\sqrt{N\eta}}\Psi_1(s)+\frac{N^\xi}{N}\Psi_1(s)+\frac{\vert I\vert^2}{N}N^{5\xi+\delta_1}\Psi_2(s,\eta)
	\right)
\right).
\end{multline*}

If we take $\vartheta_2>0$ and $\omega>0$ small enough, we can consider $\delta_2>0$ such that $\eta = s^2N^{-\delta_2}\gg N^{-2/3}$ and $z_k\in\mathcal{D}_\omega$. In this case, by the rigidity of eigenvalues from $\mathcal{A}_1$ we have that $\Im m(z_{k_i})\geqslant \sqrt{\eta}$ so that we can rewrite the previous bound as
\begin{multline*}
 	\partial_s\left(
 	f^{\mathrm{Fer}}_s(\mathbf{k^m})+\frac{\vert I\vert}{N^2}
 	\right)
 	\leqslant
 	-\frac{c}{\sqrt{\eta}}\left(
 	f^{\mathrm{Fer}}_s(\mathbf{k^m})+\frac{\vert I\vert}{N^2}
 	\right)
 	\\
 	+
 	\mathcal{O}{}\left(
 	\frac{1}{\sqrt{\eta}}\left(
 	\frac{N^{\xi}}{N\eta^{3/2}}\Psi_1^2(s)+\frac{N^\xi\vert I\vert}{N^{3/2}\eta}\Psi_1(s)+\frac{N^\xi}{N\sqrt{\eta}}\Psi_1(s)+\frac{\vert I\vert^2}{N\sqrt{\eta}}N^{5\xi+\delta_1}\Psi_2(s,\eta)
 	\right)
 	\right).
 \end{multline*}
By Gronwall's lemma, we then obtain for any $D>0$,
\begin{multline*}
f^{\mathrm{Fer}}_s(\mathbf{k^m})+\frac{\vert I\vert}{N^2}
=
\O{
\frac{N^{\xi+3\delta_2/2}}{Ns^{3/2}}\Psi_1^2(s)
+
\frac{N^{\xi+\delta_2}\vert I\vert}{N^{3/2}s}\Psi_1(s)
+
\frac{N^{\xi+\delta_2/2}}{N\sqrt{s}}\Psi_1(s)
+
\right.\\
\left.
\frac{\vert I\vert^2N^{5\xi+\delta_1+\delta_2/2}}{N\sqrt{s}}\Psi_2(s,N^{-\delta_2}s)
+
N^{-D}
}.
\end{multline*}  
We can then consider $\theta_2(\vartheta),\xi(\vartheta)$ and $\delta_1(\vartheta)$ sufficiently small such that for $s\in [N^{-\theta_2},1]$ there exists $\vartheta_2(\xi, \delta_1, \delta_2, s)>0$ such that 
\[
f^{\mathrm{Fer}}_s
(\mathbf{k^m})=-\frac{\vert I\vert}{N^2}+\mathcal{O}\left(\frac{\vert I\vert}{N^2}N^{-\vartheta_2}\right).
\]
Thus, the case $n=2$ has been proved and we obtain the correct value as in Lemma \ref{lem:detgauss}.
\end{proof}
\subsection{Fermionic observable with \texorpdfstring{$\bm{n}$}{n} particles} 
We show in this subsection that our determinant is asymptotically close to $A_n$ and thus confirming the idea that, in the sense of moments, the family of $(p_{k\ell})$ are independent Gaussian for broader moments than just correlations. Unfortunately, the knowledge of these moments does not seem enough to say that the whole family behaves in this fashion. To obtain Theorem \ref{theo:mainresult}, we only need the previous subsection with two particles but we state here the theorem for any value of $n$. Note that we need here stronger assumptions on our observable as we consider $\vert I\vert\ll \sqrt{N}$ instead of the optimal $\vert I\vert\ll N$. This is due to the fact that there does not seem to be a direct way to generalize the proof of Proposition \eqref{prop:decorr} to $n$ particles.

\begin{theorem}\label{theo:resultinterm}
Let $n\in\mathbb{N}.$ Let $\vartheta\in(0,\frac{1}{4})$ and $I\subset\unn{1}{N}$ such that $N^{\vartheta}\leqslant \vert I\vert \leqslant N^{1/2-\vartheta}$. We can take $\xi(n,\vartheta),\delta_1(n,\vartheta), \omega(n,\vartheta)>0$ sufficiently small such that if $f_s^{\mathrm{Fer}}$ is as in \eqref{eq:fermion}, conditionally on $H_0,\bm{\lambda}\in\mathcal{A}_2(\xi,\delta_1,\omega)$, there exists $\theta_n=\theta_n(n,\vartheta,\xi,\delta_1,\omega)$  sufficiently small such that for $N^{-\theta_n}\leqslant s\leqslant 1$ there exists a positive $\vartheta_n=\vartheta_2(n,s,\xi\,\delta_1,\omega,\vartheta)$ such that
\begin{equation}\label{eq:resultinterm}
\sup_{\mathbf{k}, \vert \mathbf{k}\vert=n}
\left\vert
	f_s^{\mathrm{Fer}}(\mathbf{k})
	-
	\left(
		\frac{\sqrt{\vert I\vert}}{N}
	\right)^{n}A_n
\right\vert =\mathcal{O}{}\left(
	\left(
		\frac{\sqrt{\vert I\vert}}{N}
	\right)^{n}
	N^{-\vartheta_n}
\right). 
\end{equation}
\end{theorem}
Note that in the proof we always do a maximum principle in order to obtain our leading order but the same estimates can be done on the infimum of our observable so that we get our result.
\begin{proof}[Proof of Theorem \ref{theo:resultinterm}]
We reason by induction and use a recursion formula in order to obtain the value of our Fermionic observable, we thus need to obtain an estimate on the observable for small $n$ (the size of the determinant). We already obtained an estimate for $n=2$ (which also holds in the particular case of $\vert I\vert\leqslant N^{1/2-\vartheta}$). We now describe the estimate we need for $n=1$. 

Let $\xi>0$. In this case we have that
$f^{\mathrm{Fer}}_s(k)=\mathds{E}\left[
	p_{kk}\middle\vert\bm{\lambda} 
\right]
$. We can obtain an estimate by using a maximum principle on $f^{\mathrm{Fer}}_s$. Consider $k_m$ the index, chosen similarly as in the proofs above, such that 
\[
f^{\mathrm{Fer}}_s(k_m)=\sup_{k\in\unn{1}{N}} f^{\mathrm{Fer}}_s(k).
\]
Then we have, since $f_s^{\mathrm{Fer}}$ follows the dynamics $\eqref{eq:emf}$, that for any $\eta>0,$
\[
\partial_s f_s^{\mathrm{Fer}}(k_m)=2\sum_{\ell\neq k_m}\frac{f^{\mathrm{Fer}}_s(\ell)-f_s^{\mathrm{Fer}}(k_m)}{N(\lambda_\ell-\lambda_{k_m})^2}
\leqslant
\frac{2}{\eta}\mathds{E}\left[
	\frac{1}{N}
	\sum_{\ell\neq k_m}
\frac{(p_{\ell\ell}-p_{k_mk_m})\eta}{(\lambda_\ell-\lambda_{k_m})^2+\eta^2}
	\middle\vert\bm{\lambda}
\right].
\]
Now, one can see that 
\[
p_{k_mk_m}\frac{1}{N}\sum_{\ell\neq k_m}\frac{\eta}{(\lambda_\ell-\lambda_{k_m})^2+\eta^2}=p_{k_mk_m}\Im m(z_{k_m})+\mathcal{O}{}\left(\frac{N^\xi}{N\eta}\Psi_1(s)\right)
\]
where we reintroduced the notation $z_{k_i}=\lambda_{k_i}+\I\eta$. For the other term, we use the entryise local law from \eqref{eq:isolocal} from the event $\mathcal{A}_1$,
\begin{multline*}
\frac{1}{N}\sum_{\ell\neq k_m}\frac{p_{\ell\ell}\eta}{(\lambda_\ell-\lambda_{k_m})^2+\eta^2}=
\frac{1}{N}\sum_{\alpha\in I}G_{\alpha\alpha}(z_{k_m})-\frac{\vert I\vert}{N}\Im m(z_{k_m})+\mathcal{O}{}\left(
	\frac{N^\xi}{N\eta}\Psi_1(s)
\right)\\
=
\mathcal{O}{}\left(
	\frac{N^\xi\vert I\vert}{N\sqrt{N\eta}}+\frac{N^\xi}{N\eta}\Psi_1(s)
\right).
\end{multline*}
Thus, we obtain the following Gronwall type inequality,
\[
\partial_s f^{\mathrm{Fer}}_s(k_m)
\leqslant
-\frac{2\Im m(z_{k_m})}{\eta}f^{\mathrm{Fer}}_s(k_m)+\mathcal{O}{}\left(
\frac{\vert I\vert N^\xi}{N\eta\sqrt{N\eta}}+\frac{N^\xi}{N\eta^2}\Psi_1(s)
\right)
\]
which gives us that, as long as we consider $\theta_1, \delta_2, \omega$  small enough, we have  $\eta=s^2N^{-\delta_2}\gg N^{-2/3}$, 
\[
f_s^{\mathrm{Fer}}(k_m)=\mathcal{O}{}\left(
	\frac{\vert I\vert N^\xi}{N^{3/2}\eta}+\frac{N^\xi}{N\eta^{3/2}}\Psi_1(s)
\right).
\]
Thus, we can choose $\theta_1, \xi,\delta_2$ small enough to obtain an estimate such that the error is smaller than $\sqrt{\vert I\vert}/N$ for $s\in[N^{-\theta_1},1]$. w
e need to consider $s$ such that 
Note that the choices of our parameters is consistent since we consider $\vert I\vert\ll \sqrt{N}\ll N$. The condition $\vert I\vert\ll \sqrt{N}$ is actually only needed for the $n\geqslant 3$ particles case. Thus, the case $n=1$ goes in the direction of Lemma \ref{lem:detgauss}.

Consider now the case where $n$ is an integer greater than 2. For the general case, we develop our Fermionic observable via the Leibniz formula, for $\mathbf{k}$ such that $\vert \mathbf{k}\vert=n$ and $\mathfrak{S}_n$ the set of permutations of $\unn{1}{n}$,
\[
f^{\mathrm{Fer}}_s(\mathbf{k})
=
\mathds{E}\left[
	\det P_s(\mathbf{k})
	\middle\vert
	\bm{\lambda}
\right]=\sum_{\sigma\in \mathfrak{S}_n}\epsilon(\sigma)\mathds{E}\left[
	\prod_{i=1}^n p_{k_ik_{\sigma(i)}}(s)
	\middle\vert
	\bm{\lambda}
\right].
\]
As earlier, we use a maximum principle technique in order to obtain the leading order for Theorem \ref{theo:resultinterm}. Consider $\mathbf{k^m}$ maximizing $f^\mathrm{Fer}_s(\mathbf{k})$ chosen as before and write
\begin{multline}\label{eq:interm2}
\partial_s f_s^{\mathrm{Fer}}(\mathbf{k^m})
=
2\sum_{i=1}^n
\sum_{\ell\notin\{k_1^m,\dots,k_n^m\}}\frac{f_s^{\mathrm{Fer}}((\mathbf{k^m})^i(\ell))-f_s^{\mathrm{Fer}}(\mathbf{k^m})}{N(\lambda_{k_i^m}-\lambda_\ell)^2}
\\
\leqslant
\frac{2}{\eta}\sum_{i=1}^n
\!\sum_{\ell\notin\{k_1^m,\dots,k_n^m\}}\frac{(f_s^{\mathrm{Fer}}((\mathbf{k^m})^i(\ell))-f_s^{\mathrm{Fer}}(\mathbf{k^m}))\eta}{N((\lambda_{k_i^m}-\lambda_\ell)^2+\eta^2)}.
\end{multline}
Now, we can also write that since $n$ is fixed independent of $N$,
\[
\frac{1}{N}\sum_{\ell\notin \{k_1^m,\dots,k_n^m\}}f_s^{\mathrm{Fer}}(\mathbf{k^m})\frac{\eta}{(\lambda_{k_i^m}-\lambda_\ell)^2+\eta^2}
=
f_s^{\mathrm{Fer}}(\mathbf{k^m})\Im m(z_{k_i^m})
+
\mathcal{O}{}\left(
	\frac{N^{(n+1)\xi}}{N\eta}\Psi_1^n(s)
\right)
\]
where we used the fact that $f_s^{\mathrm{Fer}}(\mathbf{k^m})\leqslant N^{n\xi}\Psi_1^n(s)$ and the local law \eqref{eq:avelocal}. 

In order to control $f^{\mathrm{Fer}}_s((\mathbf{k^m})^i(\ell))$, we partition $\mathfrak{S}_n$ into three sets which give different contributions to the result, note that we make the permutations on the set given by the indices in $(\mathbf{k^m})^i(\ell)$ but since the number of indices stay constant and is equal to $n$, this dependence does not matter in our computations,
\begin{align*}
\mathfrak{S}^{(1)}_n(i)&=\left\{
	\sigma\in \mathfrak{S}_n:\, \sigma(i)=i
\right\},\\
\mathfrak{S}^{(2)}_n(i)&=\left\{
	\sigma\in \mathfrak{S}_n:\, \sigma(i)=\sigma^{-1}(i)\quad\text{and}\quad \sigma(i)\neq i
\right\},\\
\mathfrak{S}^{(3)}_n(i)&=\mathfrak{S}_n\setminus(\mathfrak{S}_n^{(1)}\sqcup\mathfrak{S}_n^{(2)}).
\end{align*}
Now see that we can write, for a fixed $i\in\unn{1}{n},$
\[
\frac{1}{N}\sum_{\ell\notin\{k_1^m,\dots,k_n^m\}}\frac{\eta f_s^{\mathrm{Fer}}((\mathbf{k^m})^i(\ell))}{(\lambda_{k_i^m}-\lambda_\ell)^2+\eta^2}=\frac{1}{N}\Im \sum_{\ell=1}^N\frac{f_s^{\mathrm{Fer}}((\mathbf{k^m})^i(\ell))}{(\lambda_{\ell}-z_{k_i^m})}
+
\mathcal{O}{}\left(
	\frac{N^{n\xi}}{N\eta}\Psi^n_1(s)
\right).
\]
By developing $f_s^{\mathrm{Fer}}((\mathbf{k^m})^i(\ell))$ according to the Leibniz formula and separating this sum in three terms with respect to the prior partition of $\mathfrak{S}_n$, we now have to control three terms. The first one can be written as
\[
\mathrm{(I)}:=\sum_{\sigma\in\mathfrak{S}^{(1)}_n(i)}\epsilon(\sigma)
\Im \sum_{\ell=1}^N\frac{p_{\ell\ell}}{N(\lambda_\ell-z_{k_i})}
\prod_{\substack{j=1\\j\neq i}}^np_{k_jk_{\sigma(j)}}
=
\mathcal{O}\left(
	\frac{N^{n\xi}\vert I\vert}{N\sqrt{N\eta}}\Psi_1^{n-1}(s)
\right)
\]
using the local law from \eqref{eq:isolocal} and Proposition \ref{lem:queedge}. Now, for the contribution of $\mathfrak{S}_n^{(2)}$ we have to control
\begin{equation*}
\mathrm{(II)}
:=
\sum_{\sigma\in\mathfrak{S}_n^{(2)}(i)}\epsilon(\sigma)
\Im\sum_{\ell=1}^N\frac{p_{\ell k_{\sigma(i)}}^2}{N(\lambda_{\ell}-z_{k_i})}
\prod_{\substack{j\neq i\\j\neq \sigma(i)}}p_{k_jk_{\sigma(j)}}.
\end{equation*}
Recall the definition of $p_{\ell k_i}$ \eqref{eq:defpkk2}, we can write the previous term as a resolvent in the following way
\begin{multline*}
\Im\sum_{\ell=1}^N\frac{p_{\ell k_{\sigma(i)}}^2}{N(\lambda_{\ell}-z_{k_i})}
=
\frac{1}{N}
\sum_{\alpha,\beta\in I}
u_{k_{\sigma(i)}}(\alpha)
u_{k_{\sigma(i)}}(\beta)
\Im G_{\alpha\beta}(z_{k_i})
\\=
\frac{\vert I\vert}{N^2}\Im m(z_{k_i})
+
\O{
	\frac{N^\xi\Psi_1(s)}{N}
	+
	\frac{N^{\xi}\vert I\vert}{N^2\sqrt{N\eta}}
	+
	\frac{N^{2\xi}\vert I\vert^2}{N^2\sqrt{N\eta}}
}.
\end{multline*}
The leading contribution in the previous equation comes from the diagonal terms in the sum. The first error term comes from applying Proposition \ref{lem:queedge}, the second from applying \eqref{eq:isolocal} and the last one from combining \eqref{eq:deloc} and \eqref{eq:isolocal} to bound the off-diagonal terms. Note that the second error term is always smaller than the last one.

 It is possible to see (II) as a sum over the possible $\sigma(i)$ in the product so that we can write it as a sum of determinant of size $n-2$ in order to conclude later by induction. Indeed, we have
\[
\mathrm{(II)}=-\frac{\vert I\vert}{N^2}\Im m(z_{k_i})\sum_{\substack{i_0=1\\i_0\neq i}}^n\det {P_s}_{(i,i_0)}^{(i,i_0)}
+\mathcal{O}{}\left(
	\left(
		\frac{N^\xi\Psi_1(s)}{N}
		+
		\frac{N^{2\xi}\vert I\vert^2}{N^2\sqrt{N\eta}}
	\right)
	N^{(n-2)\xi}\Psi_1^{n-2}(s).
\right).
\]
Note that in the previous equation we obtain a minus sign from the signatures of the permutations. Indeed, as the estimate removed the cycle $(k_ik_{\sigma(i)})$, it removed two elements from the set so that if one writes the signature as $\epsilon(\sigma)=(-1)^{n-\mathscr{C}(\sigma)}$ with $\mathscr{C}(\sigma)$ the number of cycles of the permutation $\sigma$, the new signature becomes $(-1)^{n-2-\mathscr{C}(\sigma)+1}=-\epsilon(\sigma).$ It remains to bound the last term coming from $\mathfrak{S}_n^{(3)}$,
\[
\mathrm{(III)}= \sum_{\sigma\in\mathfrak{S}_n^{(3)}(i)}\epsilon(\sigma)
\Im \sum_{\ell=1}^N\frac{p_{\ell k_{\sigma(i)}}p_{k_{\sigma^{-1}(i)}\ell}}{N(\lambda_\ell-z_{k_i})}
\prod_{\substack{j\neq i\\j \neq \sigma^{-1}(j)}}p_{k_jk_{\sigma(j)}}.
\]
Now, we can write the last sum as,
\begin{multline*}
\Im \sum_{\ell=1}^N\frac{p_{\ell k_{\sigma(i)}}p_{k_{\sigma^{-1}(i)}\ell}}{N(\lambda_\ell-z_{k_i})}
=
\frac{1}{N}\sum_{\alpha,\beta\in I}u_{k_{\sigma(i)}}(\alpha)u_{k_{\sigma^{-1}(i)}}(\beta)\Im G_{\alpha\beta}(z_{k_i}) 
\\
= \frac{1}{N}\Im m(z_{k_i})p_{k_{\sigma(i)}k_{\sigma^{-1}}(i)}
+
\mathcal{O}{}\left(
	\frac{N^{2\xi}}{N\sqrt{N\eta}}\Psi_1(s)+\frac{N^{2\xi}\vert I\vert^2}{N^2\sqrt{N\eta}}
\right)
=
\mathcal{O}{}\left(
	\frac{N^{2\xi}}{N}\Psi_1(s)+\frac{N^{2\xi}\vert I\vert^2}{N^2\sqrt{N\eta}}
\right)
\end{multline*}
which gives us that 
\[
\mathrm{(III)}=\mathcal{O}{}\left(
	\frac{N^{n\xi}}{N}\Psi_1^{n-1}(s)+\frac{N^{n\xi}\vert I\vert^2}{N^2\sqrt{N\eta}}\Psi_1^{n-2}(s)
\right).
\]
Finally, putting all these estimates together in \eqref{eq:interm2} , we obtain the following inequality
\begin{multline*}
\partial_sf^{\mathrm{Fer}}_s(\mathbf{k^m})\leqslant
-C\sum_{i=1}^n
\frac{\Im m(z_{k_i})}{\eta}\left(
	f_s^{\mathrm{Fer}}(\mathbf{k}^m)+\frac{\vert I\vert}{N^2}\sum_{\substack{i_0=1\\i_0\neq i}}^n
	\mathds{E}\left[
		\det {P_s}_{(i,i_0)}^{(i,i_0)}
		\middle\vert
		\bm{\lambda}
	\right]
\right)\\
+
\mathcal{O}{}\left(
	\frac{N^{(n+1)\xi}\Psi_1^n(s)}{N\eta^2}
	+
	\frac{N^{n\xi}\vert I\vert\Psi_1^{n-1}(s)}{(N\eta)^{3/2}}
	+
	\frac{N^{n\xi}\Psi_1^{n-1}(s)}{N\eta}
	+
	\frac{N^{n\xi}\vert I\vert^2\Psi_1^{n-2}(s)}{N^{5/2}\eta^{3/2}}
\right).
\end{multline*}
Now, we are going to use our induction hypothesis, since $\mathds{E}\left[\det {P_s}_{(i,i_0)}^{(i,i_0)}\middle\vert\bm{\lambda}\right]$ corresponds to the Fermionic observable with a configuration of $n-2$ particles (we removed a particle in $i$ and in $i_0$), thus we suppose that we can consider $\theta_{n-2},\xi,\delta_1,\omega$ sufficiently small such that there exists a $\vartheta_{n-2}$ such that for any $i$ and $i_0$
\[
\mathds{E}\left[
	\det {P_s}_{(i,i_0)}^{(i,i_0)}
	\middle\vert
	\bm{\lambda}
\right]=\left(
	\frac{\sqrt{\vert I\vert}}{N}
\right)^{(n-2)}A_{n-2}
+
\mathcal{O}{}\left(
	\left(
		\frac{\sqrt{\vert I\vert}}{N}
	\right)^{n-2}
	N^{-\vartheta_{n-2}}
\right).
\]
Since we obtained the correct initial conditions earlier, we obtain that 
\begin{multline*}
\partial_s f_s^{\mathrm{Fer}}(\mathbf{k^m})\leqslant
-C\sum_{i=1}^n\frac{\Im m(z_{k_i})}{\eta}\left(
	f_s^{\mathrm{Fer}}(\mathbf{k^m})
	-
	\left(
		\frac{\sqrt{\vert I\vert}}{N}
	\right)^nA_n
\right)
\\
+
\mathcal{O}{}\left(
	\frac{N^{(n+1)\xi}\Psi_1^n(s)}{N\eta^2}
	+
	\frac{N^{n\xi}\vert I\vert\Psi_1^{n-1}(s)}{(N\eta)^{3/2}}
	+
	\frac{N^{n\xi}\Psi_1^{n-1}(s)}{N\eta}
	+
	\frac{N^{n\xi}\vert I\vert^2\Psi_1^{n-2}(s)}{N^{5/2}\eta^{3/2}}
	+
	\frac{N^{-\vartheta_{n-2}}}{\eta}
	\left(
		\frac{\sqrt{\vert I\vert}}{N}
	\right)^{n}
\right).
\end{multline*}
If we take $\theta_n>0$ and $\omega>0$ small enough, we can consider $\delta_2>0$ such that $\eta = s^2N^{-\delta_2}\gg N^{-2/3}$ and $z_k\in\mathcal{D}_\omega$ for $s\in[N^{-\theta_n},1]$. In this case, by the rigidity of eigenvalues from $\mathcal{A}_1$ we have that $\Im m(z_{k_i})\geqslant \sqrt{\eta}$. By an application of Gronwall's lemma, we have for any $D>0$,
\begin{multline*}
f_s^{\mathrm{Fer}}(\mathbf{k^m})
	=
\left(
	\frac{\sqrt{\vert I\vert}}{N}
\right)^nA_n
+
\mathcal{O}{}\left(
	\frac{N^{(n+1)\xi+3\delta_2/2}\Psi_1^n(s)}{Ns^{3}}
	+
	\frac{N^{n\xi+\delta_2}\vert I\vert\Psi_1^{n-1}(s)}{N^{3/2}s^2}
	+
	\frac{N^{n\xi+\delta_2/2}\Psi_1^{n-1}(s)}{N{s}}
	+
	\right.
	\\
	\left.
	+
	\frac{N^{n\xi+\delta_2}\vert I\vert^2\Psi_1^{n-2}(s)}{N^{5/2}s^2}
	+
	\frac{N^{-\vartheta_{n-2}+\delta_2/2}}{s}
	\left(
		\frac{\sqrt{\vert I\vert}}{N}
	\right)^{n}
	+N^{-D}
\right).
\end{multline*}
Thus for the error terms to be of order less than $(\sqrt{\vert I\vert}/N)^n$, we can consider $\theta_n(n,\vartheta),\xi(n,\vartheta)$ and $\delta_1(n,\vartheta)$ sufficiently small such that for $s\in [N^{-\theta_n},1]$ there exists $\vartheta_2(n,\xi, \delta_1, \delta_2, s)>0$ such that we obtain the result. 
Note that this choice is consistent by the additional assumption that  $\vert I\vert \ll \sqrt{N}$ for the fourth error term above.
\end{proof}

Now that we have the leading order for our Fermionic observable, we can obtain Theorem \ref{theo:mainresult} for the class of matrices given by $H_s$ for a class of $N^{-\theta}\ll s\ll 1$ for some $\theta>0$,
\begin{proposition}
There exists a (small) constant $\theta>0$ such that for any $s\in[N^{-\theta}, 1]$, Theorem \ref{theo:mainresult} holds for $H_s$.
\end{proposition}

\begin{proof}
By the analysis of the Fermionic observable in Theorem \ref{theo:resultinterm}, we know that there exists $\delta>0$ such that for any $D>0$,
\begin{equation}\label{eq:system1}
\mathds{E}\left[
	p_{kk}p_{\ell\ell}
\right]
-
\mathds{E}\left[
	p_{k\ell}^2
\right]
=
-\frac{\vert I\vert}{N^2}+\mathcal{O}\left(\frac{\vert I\vert}{N^2}N^{-\delta}+N^{-D}\right)
\end{equation}
where the expectation is taken over $H_0$ and $(B_{k\ell}(t))_{t\in[0,1]}$, the last error comes from bounding the expectation over $\mathcal{A}_1^c$.  Now, while we studied our Fermionic observable it was also possible to study the Bosonic observable from \cite{bourgade2018random} which in the case of two particles consists of, for $k$ and $\ell$ two distinct indices in $\unn{1}{N}$,
\[
f_s^{\mathrm{Bos}}(k,\ell)=\mathds{E}\left[
	p_{kk}p_{\ell\ell}
	+
	2p_{k\ell}^2
	\middle\vert
	\bm{\lambda}
\right]\quad \text{and}\quad
f_s^{\mathrm{Bos}}(k,k)=
\frac{1}{3}\mathds{E}\left[
	p_{kk}^2
	\middle\vert
	\bm{\lambda}
\right]
\]
and it follows the usual eigenvector moment flow so that by a similar analysis, we can obtain
\begin{equation}\label{eq:system2}
\mathds{E}\left[
p_{kk}p_{\ell\ell}
\right]
+
2\mathds{E}\left[
	p_{k\ell}^2
\right]=2\frac{\vert I\vert}{N^2}+\mathcal{O}\left(\frac{\vert I\vert}{N^2}N^{-\delta^\prime}\right)
\end{equation}
for some positive $\delta^\prime.$ So that, combining \eqref{eq:system1} and \eqref{eq:system2}, we obtain our result for the eigenvector of the matrix $H_s$.
\end{proof}

\section{Proof of Theorem \ref{theo:mainresult}}

We have now our result for the Gaussian divisible ensemble
\[
H_s(W)=\mathrm{e}^{-s/2}W+\sqrt{1-\mathrm{e}^{-s}}\mathrm{GOE}
\]
with $s$ a small parameter (in particular $s\leqslant N^{-\varepsilon}$ for some $\varepsilon$) and any $W$ being a generalized Wigner matrix. The point of this section is to remove the Gaussian term in order to obtain the result for our original matrix $W$. We do so by using a moment matching scheme and the density of the Gaussian divisible ensemble. The main point being that we can find a generalized Wigner matrix $W_0$ such that $H_s(W_0)$ has the same first few moments as $W$ and finish the proof by a Green function comparison theorem. We give this theorem now, a variant of \cite{knowles2013eigenvector}*{Theorem 1.10} which can be found in \cite{bourgade2017eigenvector}*{Theorem 5.2}. It needs as an assumption a level repulsion estimate. The following theorem states that the level repulsion estimate holds for generalized Wigner matrices, it can be found in \cites{erdos2015gap, bourgade2014edge}.
\begin{theorem}[\cites{erdos2015gap, bourgade2014edge}]
Consider $W$ a generalized Wigner matrix and $\lambda_1\leqslant\dots\leqslant \lambda_N$ its ordered eigenvalues. There exists $\alpha_0 >0$ such that for any $0<\alpha<\alpha_0$, there exists $\delta>0$ such that for any $E\in(-2,2)$ such that we have $\gamma_{k}\leqslant E\leqslant \gamma_{k+1}$ for some $k\in\unn{1}{N},$ we have
\[
\mathds{P}\left(
	\left\vert
		\{i,\,\lambda_i\in[E-N^{-2/3}\hat{k}^{-1/3},E+N^{-2/3}\hat{k}^{-1/3}]\}
	\right\vert
	\geqslant 2
\right)\leqslant N^{-\alpha-\delta}
\]
with $\hat{k}=\min(k,N-k+1).$
\end{theorem}
\begin{remark}
Note that this result has only been technically proved in the regime where either $\hat{k}\leqslant N^{1/4}$ for the edge case or in the bulk of the spectrum. But as remarked in \cite{bourgade2017eigenvector}, this estimate can be proved to any regime of $k$ with minor modifications in the proof. 
\end{remark}
This uniform level repulsion estimate for $W$ allows us to use the generalization of the following Green function comparison theorem
\begin{theorem}[\cite{bourgade2017eigenvector}]\label{theo:greencomp}
Consider $W$ and $W^\prime$ two generalized Wigner ensembles such that the first three moments of off-diagonal entries of $W$ and $W^\prime$ are equal and that the first two moments of diagonal entries of $W$ and $W^\prime$ are equal. Suppose also that there exists a positive $a$ such that for any $i\neq j$, 
\[
\left\vert
	\mathds{E}[w_{ij}^4]
	-
	\mathds{E}[{w^\prime_{ij}}^4]
\right\vert\leqslant N^{-2-a}.
\]
Let $\alpha>0$, then there exists $\varepsilon=\varepsilon(a)>0$ such that for any $k\in\mathbb{N}$ and any $\mathbf{q}_1,\dots,\mathbf{q}_k$ and any indices $j_1,\dots,j_k\in\unn{1}{N}$ we have
\[
\left(
	\mathds{E}^W-\mathds{E}^{W^\prime}
\right)O\left(
	N\scp{\mathbf{q}_1}{u_{j_1}}^2,\dots,N\scp{\mathbf{q}_k}{u_{j_k}}^2
\right)
=
\mathcal{O}(N^{-\varepsilon})
\]
for any smooth function $O$ with polynomial growth,
\[
\vert\partial^mO(x)\vert\leqslant C(1+\vert x\vert)^C
\]
for some $C$ and for any $m\in\mathbb{N}^k$ such that $\vert m\vert\leqslant 5$.
\end{theorem}
We can now give the proof of our main result.
\begin{proof}[Proof of Theorem \ref{theo:mainresult}]
We have proved that our result holds for any matrix $H_s(W_0)$ for any generalized Wigner matrix $W_0$ and $s\in[N^{-\theta},1]$ for some positive $\theta>0.$. Now, in order to use the Green function comparison theorem for eigenvector Theorem \ref{theo:greencomp}, we simply need to be able to construct a matrix $W_0$ such that the assumption of Theorem \ref{theo:greencomp} hold for the matrices $H_s(W_0)$ and $W$. Such a construction can be seen in \cite{erdos2011bernoulli}*{Lemma 3.4} and the result has been proved.
\end{proof}

\appendix
\section{Combinatorial proof of Theorem \ref{theo:emf}}\label{app:proof}
\sectionmark{Combinatorial proof of Theorem 3.1.10}
\tikzset{every loop/.style={min distance=10mm,in=60,out=120,looseness=10}}
\tikzset{every node/.style={draw,circle,fill=black, scale=.5}}
In this appendix, we give another proof of Theorem \ref{theo:emf} where we use the generator of the dynamics \eqref{eq:dysonvect} without any consideration of Grassmann variables. The proof is based on the expansion of the determinant and a careful bookkeeping on the action of the generator on permutations.
\begin{proof}[Proof of Theorem \ref{theo:emf}]
First define
\[
g(\bm{\xi})=\sum_{\sigma\in\mathfrak{S}_n}\epsilon(\sigma)\prod_{i=1}^np_{i_ki_{\sigma(k)}}\quad\text{so that}\quad f_s^{\mathrm{Fer}}(\bm{\xi})=\mathds{E}\left[g(\bm{\xi})\middle\vert\bm{\lambda}\right].
\]
We therefore need to show the following two equality
\begin{align}
\label{eq:nopart}X^2_{i_kj}g(\bm{\xi})&=2(g(\bm{\xi}^{i_kj})-g(\bm{\xi}))\quad\text{for}\quad k\in\{1\dots,n\},\,j\notin\{i_1,\dots,i_n\},\\
\label{eq:part}X^2_{i_ki_\ell}g(\bm{\xi})&=0\quad\text{for}\quad k,\ell\in\{1,\dots,n\}.
\end{align}
Part of the reasoning is done via induction. We first describe the proof for two particles. For simplicity, we describe one of the joint moments of the family $\left(p_{k\ell}\right)$ as a graph corresponding to a permutation in the determinant. For instance, for two particles, we have two distinct graphs.

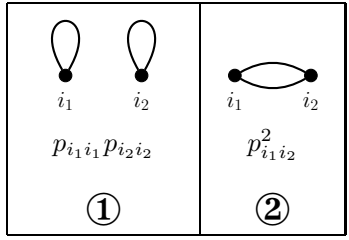
\begin{center}
\begin{tabular}{|c|c|c|}
	\hline
		\begin{tikzpicture}
			\node[label={[font=\fontsize{17.28}{0}\selectfont]below:$i_1$}] (k) at (0,0) {};
			\node[label={[font=\fontsize{17.28}{0}\selectfont]below:$i_2$}] (l) at (1,0) {};
	
			\draw[-] (k) edge[loop above, thick] (k);
			\draw[-] (l) edge[loop above, thick] (l);
		\end{tikzpicture}		
	&
		\begin{tikzpicture}
			\node[label={[font=\fontsize{17.28}{0}\selectfont]below:$i_1$}] (k) at (0,0) {};
			\node[label={[font=\fontsize{17.28}{0}\selectfont]below:$i_2$}] (l) at (1,0) {};
			
			\draw[-] (k) edge[bend left, thick] (l);
			\draw[-] (k) edge[bend right, thick] (l);
		\end{tikzpicture}
		\\
		$p_{i_1i_1}p_{i_2i_2}$
		&
		$p_{i_1i_2}^2$
		\\{}&{}
		\\
		{\LARGE \ding{172}}
		&
		{\LARGE \ding{173}}
	\\
	\hline
\end{tabular}
\end{center}
Thus we can write our fermionic observable for these two particles as
\[
g_t(\bm{\xi})=p_{i_1i_1}p_{i_2i_2}-p_{i_1i_2}^2=\text{\ding{172}}-\text{\ding{173}}.
\]
Note that we have a sign difference between the terms because of the changing signature between these two permutations. Now, see that the generator $X$ operates on our family of overlaps $(p_{i_ki_\ell})$ in the following way
\begin{equation}\label{eq:algrel}
	\begin{gathered}
		X_{i_ki_\ell}p_{i_ki_k}=-2p_{i_ki_\ell}=-X_{i_ki_\ell}p_{i_\ell i_\ell},\,X_{i_ki_\ell}p_{i_ki_\ell}=p_{i_ki_k}-p_{i_\ell i_\ell},\\
		X_{i_ki_\ell}p_{i_kj}=-p_{i_\ell j},\,X_{i_ki_\ell}p_{i_\ell j}=p_{kj}.
	\end{gathered}
\end{equation}
With these algebraic relations, one can easily see that we have \eqref{eq:nopart} for $g_t(\bm{\xi})$. One can also easily deduce \eqref{eq:part} with these simple relations, however, in order to explain the more detailed approach of the case of $n$ particles, we disclose the proof in more details. We can first operate  $X$ on both of the terms in $g_t(\bm{\xi})$ and see that 
\begin{equation}\label{eq:2part}
X^2_{i_1i_2}\left(
	p_{i_1i_1}p_{i_2i_2}
\right)=
2\left(
	p_{i_1i_1}^2+p_{i_2i_2}^2-2(p_{i_1i_1}p_{i_2i_2}+2p_{i_1i_2}^2)
\right)=X_{i_1i_2}^2p_{i_1i_2}^2.
\end{equation}
First see that \eqref{eq:2part} gives us that $X_{i_1i_2}^2g_t(\bm{\xi})=0$ and the proof for two particles is clear. However, to introduce the notations we use in the case of $n$ particles, we write \eqref{eq:2part} as the following, using  the graphical representation from the previous table,
\[
X^2_{i_1i_2}\mding{172}=
2\left(
	p_{i_1i_1}^2+p_{i_2i_2}^2-2(\mding{172}+2\mding{173})
\right)=X^2_{i_1i_2}\mding{173}.
\]

Consider now the case of $n$ particles on $\{i_1,\dots,i_n\}$. By induction, we can only look at the permutations in the sum where either $\ell(i_k)+\ell(i_\ell)= n$ or $\ell(i_k)=\ell(i_\ell)=n$ where $\ell(j)$ is the length of the cycle containing $j$. Note that the second condition is there to take in account the fact that $i_k$ and $i_\ell$ can be in the same cycle.  Also see that by definition of $X_{i_ki_\ell}$, we are only interested in the sites $i_k,\,i_{\sigma(k)},\,i_{\sigma^{-1}(k)},\,i_\ell,\,i_{\sigma(\ell)}$ and $i_{\sigma^{-1}(\ell)}.$

First consider the permutations such that $\ell(i_k)$ or $\ell(i_\ell)$ is equal to $1$, such a permutation wil be represented by \ding{172} or \ding{173} in the following table.  

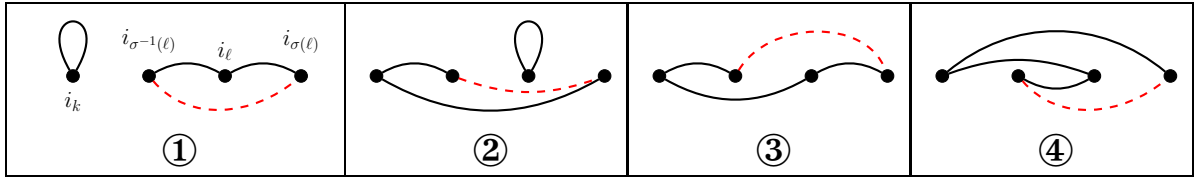
\begin{center}
\begin{tabular}{|c|c|c|c|}

	\hline
	\specialcell{
		\begin{tikzpicture}
			\node[label={[font=\fontsize{17.28}{0}\selectfont]below:$i_k$}] (1) at (0,0) {};
			\node[label={[font=\fontsize{17.28}{0}\selectfont, label distance=-.2cm]above:$i_{\sigma^{-1}(\ell)}$}] (2) at (1,0) {};
			\node[label={[font=\fontsize{17.28}{0}\selectfont]above:$i_\ell$}] (3) at (2,0) {};
			\node[label={[font=\fontsize{17.28}{0}\selectfont]above:$i_{\sigma(\ell)}$}] (4) at (3,0) {};
		
			\draw[-] (1) edge[loop above, thick] (1);
			\draw[-] (2) edge[bend left, thick] (3);
			\draw[-] (3) edge[bend left, thick] (4);
			\draw[-, red, dashed] (2) edge[out=310, in=220, thick] (4);
		\end{tikzpicture}
		\\
		\centering
		{\LARGE \ding{172}}
	}
	&
	\specialcell{
		\begin{tikzpicture}
			\node (1) at (4,0) {};	
			\node (2) at (5,0) {};
			\node (3) at (6,0) {};
			\node (4) at (7,0) {};
			\draw[-] (1) edge[bend left, thick] (2);
			\draw[-] (1) edge[bend right, thick] (4);
			\draw[-] (3) edge[loop above, thick] (3);
			\draw[-, red, dashed] (2) edge[out=340, in=200, thick] (4);
			\draw[-, opacity=0] (2) edge[out=310, in=220, thick] (4);

		\end{tikzpicture}
		\\
		\centering
		{\LARGE \ding{173}}
	}
	&
	\specialcell{
		\begin{tikzpicture}
			\node (1) at (8,0) {};
			\node (2) at (9,0) {};
			\node (3) at (10,0) {};
			\node (4) at (11,0) {};
			
			\draw[-] (1) edge[bend left, thick] (2);
			\draw[-] (1) edge[bend right, thick] (3);
			\draw[-] (3) edge[bend left, thick] (4);
			\draw[-, red, dashed] (2) edge[out = 60, in = 110, thick] (4);
			\draw[-, opacity=0] (2) edge[out=310, in=220, thick] (4);
			\draw[-, opacity=0] (3) edge[loop above, thick] (2);
		\end{tikzpicture}
		\\
		\centering
		{\LARGE \ding{174}}
	}
	&
	\specialcell{
		\begin{tikzpicture}
			\node (1) at (12,0) {};
			\node (2) at (13,0) {};
			\node (3) at (14,0) {};
			\node (4) at (15,0) {};
			
			\draw[-] (1) edge[out = 40, in = 140, thick] (4);
			\draw[-] (1) edge[out = 20, in = 160, thick] (3);
			\draw[-] (2) edge[bend right, thick] (3);
			\draw[-, red, dashed] (2) edge[out = 310, in = 220, thick] (4);
			\draw[-, opacity=0] (3) edge[loop above, thick] (3);
		\end{tikzpicture}
		\\
		\centering
		{\LARGE \ding{175}}
	}\\
	\hline
\end{tabular}
\end{center}
\paragraph{}These four graphs are the one involved when applying $X^2$ to \ding{172}. Note that while we display $i_{\sigma^{-1}(\ell)}$ and $i_{\sigma(\ell)}$ as distinct points, they could potentially be the same. The dashed red line represent the rest of the permutation, note also that we have two distinct cycles for the graphs \ding{172} and \ding{173} while there is a single cycle for the graphs \ding{174} and \ding{175} so that $\epsilon(\mding{172})=\epsilon(\mding{173})=-\epsilon(\mding{174})=-\epsilon(\mding{175})$.For simplicity, consider the notations
\begin{align*}
P^{(1)}_\ell &=p_{i_\ell i_\ell}p_{i_{\sigma^{-1}(\ell)}i_\ell}p_{i_\ell i_{\sigma(\ell)}},\\
P^{(1)}_k &=p_{i_k i_k}p_{i_{\sigma^{-1}(\ell)}i_k}p_{i_k i_{\sigma(\ell)}}.
\end{align*} 
Now, using the relations \eqref{eq:algrel} we obtain 
\[\begin{array}{ll}
X^2_{i_ki_\ell}\mding{172}
=
2\left(
	P^{(1)}_\ell+P^{(1)}_k
	-
	(2\mding{172}+2\mding{174}+2\mding{175})
\right),
&X^2_{i_ki_\ell}\mding{173}
=
2\left(
	P^{(1)}_\ell+P^{(1)}_k
	-
	(2\mding{173}+2\mding{174}+2\mding{175})
\right),\\
X^2_{i_ki_\ell}\mding{174}
=
2\left(
	P^{(1)}_\ell+P^{(1)}_k
	-
	(\mding{172}+\mding{173}+3\mding{174}+\mding{175})
\right),
&X^2_{i_ki_\ell}\mding{175}
=
2\left(
	P^{(1)}_\ell+P^{(1)}_k
	-
	(\mding{172}+\mding{173}+\mding{174}+3\mding{175})
\right).
\end{array}
\]

So that finally, taking in account the different signatures, we finally have
\[
X^2_{i_ki_\ell}\left(
	2\mding{172}+2\mding{173}-2\mding{174}-2\mding{175}
\right)
=0.
\]
Note that we have a coefficient of 2 in front of each graph because both $\sigma$ and $\sigma^{-1}$ follows the same graph.
Now, we consider permutations where $\ell(i_k)$ and $\ell(i_\ell)$ are greater than 1. Thus we consider such a permutation as the graph \ding{176} in the following table.
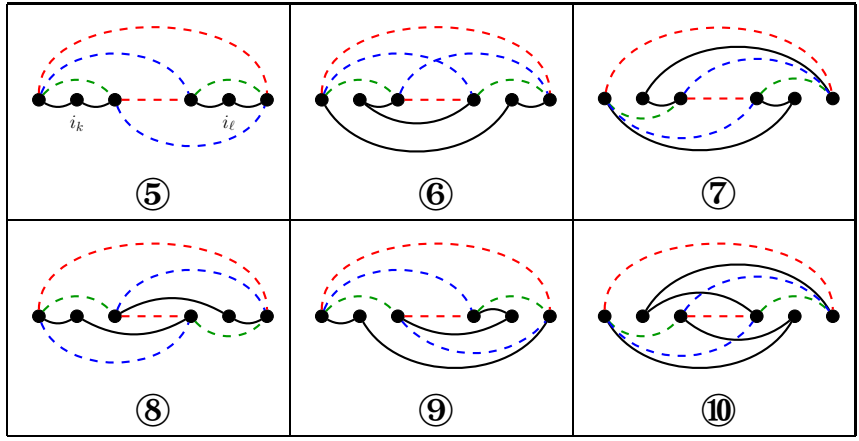
\begin{center}
\begin{tabular}{|c|c|c|}
	\hline
	\specialcell{
		\begin{tikzpicture}
			\node (1) at (0,0) {};
			\node[label={[font=\fontsize{14.4}{0}\selectfont]below:$i_k$}] (2) at (0.5,0) {};
			\node (3) at (1,0) {};
			\node (4) at (2,0) {};
		    \node[label={[font=\fontsize{14.4}{0}\selectfont]below:$i_\ell$}] (5) at (2.5,0) {};
		    \node (6) at (3,0) {};
		    
		    \draw[-] (1) edge[bend right, thick] (2);
		    \draw[-] (2) edge[bend right, thick] (3);
		    \draw[-] (4) edge[bend right, thick] (5);
		    \draw[-] (5) edge[bend right, thick] (6);
		    \draw[-,red,dashed] (3) edge[thick] (4);
		    \draw[-,red,dashed] (1) edge[out=90, in=90,thick] (6);
		    \draw[-,blue,dashed] (1) edge[out=70, in=110, thick] (4);
		    \draw[-,blue,dashed] (3) edge[out = 290, in=250, thick] (6);
		    \draw[-,black!40!green,dashed] (1) edge[out=50, in=130, thick] (3);
		    \draw[-,black!40!green,dashed] (4) edge[out=50, in=130, thick] (6);
		    \draw[-, opacity=0] (1) edge[out=300, in=240, thick] (5);

		\end{tikzpicture}
		\\
		\centering
		{\LARGE \ding{176}}
	}
	&
	\specialcell{
		\begin{tikzpicture}
			\node (1) at (4,0) {};
			\node (2) at (4.5,0) {};
			\node (3) at (5,0) {};
			\node (4) at (6,0) {};
		    \node (5) at (6.5,0) {};
		    \node (6) at (7,0) {};
		    
		    \draw[-] (1) edge[out=300, in=240, thick] (5);
		    \draw[-] (2) edge[out=320, in=220, thick] (4);
		    \draw[-] (2) edge[bend right, thick] (3);
		    \draw[-] (5) edge[bend right, thick] (6);
		    
		    \draw[-,red,dashed] (3) edge[thick] (4);
		    \draw[-,red,dashed] (1) edge[out=90, in=90,thick] (6);
		    \draw[-,blue,dashed] (1) edge[out=70, in=110, thick] (4);
		    \draw[-,blue,dashed] (3) edge[out = 70, in=110, thick] (6);
		    \draw[-,black!40!green,dashed] (1) edge[out=50, in=130, thick] (3);
		    \draw[-,black!40!green,dashed] (4) edge[out=50, in=130, thick] (6);
		\end{tikzpicture}
		\\
		\centering
		{\LARGE \ding{177}}
	}
	&
	\specialcell{
		\begin{tikzpicture}
			\node (1) at (8,0) {};
			\node (2) at (8.5,0) {};
			\node (3) at (9,0) {};
			\node (4) at (10,0) {};
		    \node (5) at (10.5,0) {};
		    \node (6) at (11,0) {};
		    
		    \draw[-] (1) edge[out=300, in=240, thick] (5);
		    \draw[-] (2) edge[bend right, thick] (3);
		    \draw[-] (4) edge[bend right, thick] (5);
		    \draw[-] (2) edge[out=60, in=120, thick] (6);
		    
		    \draw[-,red,dashed] (3) edge[thick] (4);
		    \draw[-,red,dashed] (1) edge[out=80, in=100,thick] (6);
		    \draw[-,black!40!green,dashed] (4) edge[out=50, in=130, thick] (6);
		    \draw[-,black!40!green,dashed] (1) edge[out=310, in=230, thick] (3);
		    \draw[-,blue,dashed] (1) edge[out=305, in=235, thick] (4);
		    \draw[-,blue,dashed] (3) edge[out = 55, in=125, thick] (6);
		\end{tikzpicture}
		\\
		\centering
		{\LARGE \ding{178}}
	}
	\\
	\hline
	\specialcell{
		\begin{tikzpicture}
			\node (1) at (0,0) {};
			\node (2) at (0.5,0) {};
			\node (3) at (1,0) {};
			\node (4) at (2,0) {};
		    \node (5) at (2.5,0) {};
		    \node (6) at (3,0) {};
		    
		    \draw[-] (1) edge[bend right, thick] (2);
		    \draw[-] (2) edge[bend right, thick] (4);
		    \draw[-] (3) edge[bend left, thick] (5);
		    \draw[-] (5) edge[bend right, thick] (6);
		    
		    \draw[-,red,dashed] (3) edge[thick] (4);
		    \draw[-,red,dashed] (1) edge[out=90, in=90,thick] (6);
		    \draw[-,blue,dashed] (3) edge[out=70, in=110, thick] (6);
		    \draw[-,blue,dashed] (1) edge[out = 290, in=250, thick] (4);
		    \draw[-,black!40!green,dashed] (1) edge[out=50, in=130, thick] (3);
		    \draw[-,black!40!green,dashed] (4) edge[out=310, in=230, thick] (6);
		    \draw[-, opacity=0] (1) edge[out=300, in=240, thick] (5);
		\end{tikzpicture}
		\\
		\centering
		{\LARGE \ding{179}}
	}
	&
	\specialcell{
		\begin{tikzpicture}
			\node (1) at (4,0) {};
			\node (2) at (4.5,0) {};
			\node (3) at (5,0) {};
			\node (4) at (6,0) {};
		    \node (5) at (6.5,0) {};
		    \node (6) at (7,0) {};
		    		    
		    \draw[-] (1) edge[bend right, thick] (2);
		    \draw[-] (2) edge[out=300, in=240, thick] (6);
		    \draw[-] (3) edge[bend right, thick] (5);
		    \draw[-] (5) edge[bend right, thick] (4);
		    
		    \draw[-,red,dashed] (3) edge[thick] (4);
		    \draw[-,red,dashed] (1) edge[out=90, in=90,thick] (6);
		    \draw[-,blue,dashed] (1) edge[out=70, in=110, thick] (4);
		    \draw[-,blue,dashed] (3) edge[out = 310, in=230, thick] (6);
		    \draw[-,black!40!green,dashed] (1) edge[out=50, in=130, thick] (3);
		    \draw[-,black!40!green,dashed] (4) edge[out=50, in=130, thick] (6);
		\end{tikzpicture}
		\\
		\centering
		{\LARGE \ding{180}}
	}
	&
	\specialcell{
		\begin{tikzpicture}
			\node (1) at (8,0) {};
			\node (2) at (8.5,0) {};
			\node (3) at (9,0) {};
			\node (4) at (10,0) {};
		    \node (5) at (10.5,0) {};
		    \node (6) at (11,0) {};
		    
		    \draw[-] (1) edge[out=300, in=240, thick] (5);
		    \draw[-] (2) edge[out=60, in=120, thick] (6);
		    \draw[-] (3) edge[out=320, in=220, thick] (5);
		    \draw[-] (2) edge[out=40, in=140, thick] (4);
		    
		    \draw[-,red,dashed] (3) edge[thick] (4);
		    \draw[-,red,dashed] (1) edge[out=90, in=90,thick] (6);
		    \draw[-,black!40!green,dashed] (4) edge[out=50, in=130, thick] (6);
		    \draw[-,black!40!green,dashed] (1) edge[out=310, in=230, thick] (3);
		    \draw[-,blue,dashed] (1) edge[out=305, in=235, thick] (4);
		    \draw[-,blue,dashed] (3) edge[out = 55, in=125, thick] (6);
		\end{tikzpicture}
		\\
		\centering
		{\LARGE \ding{181}}
	}\\
	\hline
\end{tabular}
\end{center}
\paragraph{}In this table, we represented all the permutations which are relevant when applying $X_{i_ki_\ell}$ to a general permutation of type \ding{176}. The different colors explains the different behavior of the permutation on the rest of the sites that are not seen by the operator $X$ but are relevant when counting the signatures on the different graphs. We first introduce the following notations as earlier 
\begin{align*}
P_\ell^{(2)}&=
p_{i_{\sigma^{-1}(k)}i_\ell}
p_{i_\ell i_{\sigma(k)}}
p_{i_{\sigma^{-1}(\ell)}i_\ell}
p_{i_\ell i_{\sigma(\ell)}},\\
P_k^{(2)}&=
p_{i_{\sigma^{-1}(k)}i_k}
p_{i_ki_{\sigma(k)}}
p_{i_{\sigma^{-1}(\ell)}i_k}
p_{i_k i_{\sigma(\ell)}}.
\end{align*}
Now, if we apply $X_{i_ki_\ell}$ to a permutation such that the cycle of $i_k$ and of $i_\ell$ are greater than 1 we obtain the following set of equations:
\[
\begin{array}{ll}
X^2_{i_ki_\ell}\mding{176}
=
2\left(
	P_k^{(2)}+P_\ell^{(2)}-(2\mding{176}+\mding{177}+\mding{178}+\mding{179}+\mding{180})
\right),
\\
X^2_{i_ki_\ell}\mding{177}
=
2\left(
	P_k^{(2)}+P_\ell^{(2)}-(2\mding{177}+\mding{176}+\mding{178}+\mding{179}+\mding{181})
\right),
\\
X^2_{i_ki_\ell}\mding{178}
=
2\left(
	P_k^{(2)}+P_\ell^{(2)}-(2\mding{178}+\mding{176}+\mding{177}+\mding{180}+\mding{181})
\right),
\\
X^2_{i_ki_\ell}\mding{179}
=
2\left(
	P_k^{(2)}+P_\ell^{(2)}-(2\mding{179}+\mding{176}+\mding{177}+\mding{180}+\mding{181})
\right),
\\
X^2_{i_ki_\ell}\mding{180}
=
2\left(
	P_k^{(2)}+P_\ell^{(2)}-(2\mding{180}+\mding{176}+\mding{178}+\mding{179}+\mding{180})
\right),
\\
X^2_{i_ki_\ell}\mding{181}
=
2\left(
	P_k^{(2)}+P_\ell^{(2)}-(2\mding{181}+\mding{177}+\mding{178}+\mding{179}+\mding{181})
\right),
\end{array}
\]
Now, in order to put all these equations together, one needs to see the number of permutations following these graphs and their respective signature. Both of these values depend on the number of cycles, which is equal to 1 or 2 in these cases, of the permutations and thus depend on the corresponding color in the previous table. We finally have
\[\arraycolsep=10pt\def\arraystretch{1.4}
\begin{array}{ll}
\text{Green case:}&\epsilon(\mding{176})X^2_{i_ki_\ell}(2\mding{176}-\mding{177}-\mding{178}-\mding{179}-\mding{180}+2\mding{181})=0.
\\
\text{Red Case:}&\epsilon(\mding{177})X^2_{i_ki_\ell}(2\mding{177}-\mding{176}-\mding{178}-\mding{179}+2\mding{180}-\mding{181})=0.
\\
\text{Blue Case:}&\epsilon(\mding{178})X^2_{i_ki_\ell}(2\mding{178}-\mding{176}-\mding{177}+2\mding{179}-\mding{180}-\mding{181})=0.
\end{array}
\]

Combining this result with the case where $\ell(i_k)$ or $\ell(i_\ell)$ is equal to 1 gives us the result for any permutation which finally gives
\[
X^2_{i_ki_\ell}g_t(\bm{\xi})=0.
\]
\end{proof}

\section{The Hermitian case}\label{app:hermi}
In this paper, we focused and developed the proof for symmetric random matrices, but the proof holds for Hermitian matrices as well. While the maximum principle technique can clearly be directly applied to the Hermitian case, we focused here in the definition of the Fermionic observable for the Hermitian Dyson Brownian motion. The Dyson vector flow in this case has a different generator and it is not necessarily clear that the determinant is still the correct one. Indeed, the Bosonic observable has a different form for Hermitian matrices \cite{bourgade2018random}*{Appendix} since we obtain the permanent of a matrix instead of a Hafnian. We now give the Dyson flow of eigenvalues and eigenvectors for Hermitian matrices.

\begin{definition}\label{def:dysonh}
Let $B$ be a Hermitian $N\times N$ matrix such that $\Re B_{ij},\Im B_{ij}$ for $i<j$ and $B_{ii}/\sqrt{2}$ are standard independent brownian motions. The Hermitian Dyson Brownian motion is given by the stochastic differential equation 
\begin{equation}\label{eq:dysonh}
\D H_s=\frac{\D B_s}{\sqrt{2N}}-\frac{1}{2} H_s\D t.
\end{equation}
Besides, it induces the following dynamics on eigenvalues and eigenvectors,
\begin{align}
\mathrm{d}\lambda_k&=\frac{\mathrm{d}\widetilde{B}_{kk}}{\sqrt{2N}}+\left(
	\frac{1}{N}\sum_{\ell\neq k}\frac{1}{\lambda_k-\lambda_\ell}-\frac{\lambda_k}{2}
\right)\D s,\\
\mathrm{d}u_k&=\frac{1}{\sqrt{2N}}\sum_{\ell\neq k}\frac{\mathrm{d}\widetilde{B}_{k\ell}}{\lambda_k-\lambda_\ell}u_\ell-\frac{1}{2N}\sum_{\ell\neq k}\frac{\mathrm{d}s}{(\lambda_k-\lambda_\ell)^2}u_k\label{eq:dysonvecth}
\end{align}
where $\widetilde{B}$ is distributed as $B$.
\end{definition}
The generator for the Hermitian Dyson vector flow is also known and given in the following proposition.

\begin{proposition}[\cite{bourgade2017eigenvector}]\label{prop:generatorh}
The generator acting on smooth functions of the diffusion \eqref{eq:dysonvecth} is given by 
\begin{equation}
L_t = \frac{1}{2}\sum_{1\leqslant k<\ell\leqslant N}\frac{1}{N(\lambda_k-\lambda_\ell)^2}
\left(
	X_{k\ell}\overline{X}_{k\ell}+\overline{X}_{k\ell}X_{k\ell}
\right)
\end{equation}
with the operator $X_{k\ell}$ defined by
\begin{equation*}
X_{k\ell}=\sum_{\alpha=1}^N
\left(
	u_k(\alpha)\partial_{u_\ell(\alpha)}
	-
	\overline{u}_\ell(\alpha)\partial_{\overline{u}_\ell(\alpha)}
\right)
\quad
\text{and}
\quad
\overline{X}_{k\ell}=\sum_{\alpha=1}^N
\left(
	\overline{u}_k(\alpha)\partial_{\overline{u}_\ell(\alpha)}
	-
	{u}_\ell(\alpha)\partial_{{u}_\ell(\alpha)}
\right).
\end{equation*}
\end{proposition}
We see that the determinant of fluctuations is again an observable which follows the Fermionic eigenvector moment flow. In the Hermitian case, if one considers $(u_1,\dots,u_N)$ the eigenvectors associated to the eigenvalues $\lambda_1\leqslant \dots\leqslant\lambda_N$ of $H_s$ given by \eqref{eq:dysonh}, we define the fluctuations and mixed overlap by, for a family $( \mathbf{q}_\alpha)_{\alpha\in I}\in (\mathbb{R}^N)^{\vert I\vert}$,
\[
p_{kk}=\sum_{\alpha\in I}\vert \scp{\mathbf{q}_\alpha}{u_k}\vert^2-\frac{\vert I\vert}{2N}
\quad\text{and}\quad
p_{k\ell}=\sum_{\alpha\in I}\scp{\mathbf{q}_\alpha}{u_k}\scp{\mathbf{q}_\alpha}{\overline{u}_\ell}
\hspace{.5em}\text{for}\hspace{.5em} k\neq \ell.
\]
Note in particular that we have $p_{k\ell}\neq p_{\ell k}$ but $p_{k\ell}=\overline{p_{\ell k}}.$ Now, we define the same observable, for $\mathbf{k}=(k_1,\dots,k_n)$, with $k_i\neq k_j$,
\begin{equation}\label{eq:detherm}
f^{\mathrm{Fer}}_s(\mathbf{k})=\mathds{E}\left[
	\det P_s(\mathbf{k})
	\middle\vert
	\bm{\lambda}
\right]
\end{equation}
with $P_s(\mathbf{k})$ given by \eqref{eq:defmatrix}, note that it becomes a Hermitian matrix instead of a symmetric matrix in the symmetric case. We then have the same fact that $f_s^{\mathrm{Fer}}$ follows the eigenvector moment flow.
\begin{theorem}\label{theo:emfh}
Let $(\mathbf{u},\bm{\lambda})$ be the solution to the coupled flows as in Definition \ref{def:dysonh} and let $f_{s}^{\mathrm{Fer}}$ be as in \eqref{eq:detherm}, it satisfies the following equation, for $\mathbf{k}$ a pairwise distinct set of indices such that $\vert\mathbf{k}\vert=n$,
\begin{equation}\label{eq:emfh}
\partial_s f_s^{\mathrm{Fer}}(\mathbf{k})=\sum_{i=1}^n\sum_{\substack{\ell\in[\![1,N]\!]\\\ell\notin\{k_1,\dots,k_n\}}}
\frac{f_s^{\mathrm{Fer}}(\mathbf{k}^i(\ell))-f_s^{\mathrm{Fer}}(\mathbf{k})}{N(\lambda_{k_i}-\lambda_\ell)^2}.
\end{equation}
\end{theorem}

The proof of Theorem \ref{theo:emfh} can also be done using Grassmann variables and a Fermionic Wick theorem as in Section \ref{sec:grassmann} or by carefully expanding the determinant and following the contribution of each permutation as in Appendix \ref{app:proof}. We do not develop the proof here as it is very similar but it is interesting to note that the determinant and the Fermionic eigenvector moment flow is universal regarding the symmetry of the system contrary to the Bosonic observable. Indeed, we saw the definition of the Bosonic observable via \eqref{eq:perfobs} for the symmetric Dyson flow, but the Bosonic observable in the Hermitian case is different. 

While we can also define it as a sum over (colored) graphs similarly to \eqref{eq:perfobs} another possible definition can be given in the following way: Let $\bm{\xi}$ be a configuration of $n$ particles, denote the position of the sites where each particle is situated as $(k_1,\dots,k_n)$  (note that we can have $k_i=k_j$ for some $i$'s and $j$'s) then we can define
\[
f^{\mathrm{Bos}}(\bm{\xi})=\frac{1}{\mathcal{M}(\bm{\xi})}\mathds{E}\left[
	\per P_s(\bm{\xi})
	\middle\vert
	\bm\lambda
\right]
\quad\text{with}\quad P_s(\bm{\xi})=\left(
p_{k_ik_j}
\right)_{1\leqslant i,j\leqslant n}
\quad\text{and}\quad
\mathcal{M}(\bm{\xi})=\prod_{i=1}^N \eta_i!
\]
where $\per$ denote the permanent of the matrix,
\[
\per A=\sum_{\sigma\in\mathfrak{S}_n} \prod_{i=1}^nA_{i,\sigma(i)}.
\]

\end{document}